\documentclass[reqno, 11pt]{amsart}

\usepackage[english]{babel}
\usepackage{amsmath}
\usepackage{amssymb}
\usepackage{mathrsfs}
\usepackage{enumerate}
\usepackage{ifthen}
\usepackage{bbm}
\usepackage{color}
\usepackage{graphicx}
\usepackage{geometry}
\provideboolean{shownotes}
\setboolean{shownotes}{true}
\usepackage{color}
\usepackage{setspace}
\usepackage{mathtools}
\usepackage{multicol}
\usepackage[T1]{fontenc}
\usepackage[utf8]{inputenc}
\usepackage{comment}
\usepackage{thmtools}

\newtheorem{theorem}{Theorem}[section]
\newtheorem{proposition}[theorem]{Proposition}
\newtheorem{lemma}[theorem]{Lemma}

\newtheorem{definition}[theorem]{Definition}

\theoremstyle{remark}
\newtheorem{remark}[theorem]{Remark}

\newcommand{\e}{\varepsilon}
\newcommand{\R}{\mathbb{R}}

\newcommand{\Z}{\mathbb{Z}}
\newcommand{\dive}{\mathop{\mathrm {div}}}
\newcommand{\curl}{\mathop{\mathrm {curl}}}

\newcommand{\de}{\,\mathrm{d}}
\def\cq{{\mathcal Q}}
\def\cp{{\mathcal P}}
\DeclareMathOperator*{\esssup}{ess\,sup}
\def\xp{{\tilde{\mathcal X}}}
\def\yp{{\tilde{\mathcal Y}}}
\def\zp{{\tilde{\mathcal Z}}}
\def\wp{{\tilde{\mathcal W}}}
\def\vp{{\tilde{\mathcal V}}}

\numberwithin{equation}{section}

\usepackage{hyperref}

\allowdisplaybreaks

\begin{document}

\title[Large bulk viscosity limit for compressible MHD equations]{Large bulk viscosity limit for compressible MHD equations in critical Besov spaces}

%

\author[G. Ciampa]{Gennaro Ciampa}
\address[G.\ Ciampa]{DISIM - Dipartimento di Ingegneria e Scienze dell'Informazione e Matematica\\ Universit\`a  degli Studi dell'Aquila \\Via Vetoio \\ 67100 L'Aquila \\ Italy}
\email[]{\href{gciampa@}{gennaro.ciampa@univaq.it}}

\author[D. Donatelli]{Donatella Donatelli}
\address[D.\ Donatelli]{DISIM - Dipartimento di Ingegneria e Scienze dell'Informazione e Matematica\\ Universit\`a  degli Studi dell'Aquila \\Via Vetoio \\ 67100 L'Aquila \\ Italy}
\email[]{\href{ddonatelli@}{donatella.donatelli@univaq.it}}

\author[G. Pellecchia]{Giada Pellecchia}
\address[G.\ Pellecchia]{DISIM - Dipartimento di Ingegneria e Scienze dell'Informazione e Matematica\\ Universit\`a  degli Studi dell'Aquila \\Via Vetoio \\ 67100 L'Aquila \\ Italy}
\email[]{\href{gpellecchia@}{giada.pellecchia@graduate.univaq.it}}


\begin{abstract}
We study the large bulk viscosity limit for the compressible magnetohydrodynamics (MHD) equations in two and three dimensions. For arbitrarily large initial data in critical Besov spaces, we prove the global well-posedness of strong solutions and establish their convergence, with explicit quantitative rates, to solutions of the incompressible MHD system, as the bulk viscosity parameter tends to infinity. As an application of this singular-limit analysis, we construct global smooth solutions to the compressible MHD equations whose magnetic field undergoes reconnection, thereby extending to the compressible regime the reconnection scenarios previously identified for incompressible flows.
\end{abstract}

\maketitle

\section{Introduction}

We consider the Cauchy problem for the compressible Magnetohydrodynamic system
\begin{equation}\label{MHDcomp}
    \begin{cases}
        \rho_t+\dive(\rho v)=0, \\
        \rho v_t+\rho v\cdot\nabla v-\mu\Delta v-(\lambda+\mu)\nabla\dive v +\nabla P=b\cdot \nabla b-\frac{1}{2}\nabla(|b|^2), \\
         b_t+b\dive v+v\cdot\nabla b-b\cdot\nabla v=\eta\Delta b,\\
         \dive b=0,\\
         \rho|_{t=0} =\rho_0, \quad v|_{t=0}=v_0, \quad b|_{t=0}=b_0,
    \end{cases}
    \tag{C-MHD}
\end{equation}
posed in the whole space $\R^d$, with $d=2,3$. The unknowns are the density $\rho$, the fluid velocity $v$, the magnetic field $b$. We assume that the pressure function $P$ is of the form $P=P(\rho)=A\rho^\gamma$ with $\gamma>1$, so that $P$ is strictly increasing.
The data of the problem are the initial state $(\rho_0,v_0,b_0)$, with $b_0$ being a divergence-free vector field. The shear and volume viscosity coefficients $\lambda$ and $\mu$ satisfy the standard strong parabolicity assumption
\begin{equation}
\mu>0\quad\text{and}\quad \nu:=\lambda+2\mu>0,
\end{equation}
and $\eta>0$ is the magnetic diffusivity. The equations \eqref{MHDcomp} describe the dynamics of an electrically conducting fluid (such as the solar wind, the heliosphere plasma, or the Earth’s magnetosphere) under the influence of a magnetic field. The motion of the fluid and the magnetic field are mutually coupled. On the one hand, the magnetic field acts on the fluid through the Lorentz force, accelerating the flow in a direction perpendicular to the magnetic field. On the other hand, the motion of the fluid modifies the magnetic field via magnetic induction. Consequently, the fluid motion is governed by the Navier–Stokes equations, whose momentum balance includes the Lorentz force, and is coupled with a simplified form of Maxwell’s equations derived under the assumption of a perfectly conducting medium. For a detailed physical derivation of the \eqref{MHDcomp} system see for instance \cite{Bellan}. The mathematical analysis of the system (C-MHD) dates back to the 1980s, when Kawashima \cite{kawashima1984smooth} proved the global existence of smooth solutions in two dimensions for small perturbations of constant states. Since then, the well-posedness theory has advanced significantly, often building on techniques developed for the compressible Navier–Stokes equations. First of all, local existence of strong solutions was obtained in \cite{JPG}. Global strong solutions for small data in $H^3(\R^3)$ were later established in \cite{LY,ZZ}, while global classical solutions with small initial energy but allowing large oscillations and vacuum were obtained in \cite{LXZ-sima}. This latter was further extended in \cite{HHPZ} allowing large initial data under the assumptions $\gamma\sim 1$ and large resistivity. In two dimensions, global classical solutions with small energy and improved decay estimates were derived in \cite{LSX}. For weak solutions, global renormalized solutions for general large data were constructed in \cite{FY,HW-CMP, HW-ARMA}. In the critical Besov setting, global strong solutions near equilibrium were obtained in \cite{Hao}, while local well-posedness under a positive lower bound for the density was proved in \cite{JPG}.\\

In the present work, we are interested in initial configurations that may be arbitrarily large in critical Besov norms naturally associated with the scaling of the system. In order to make this statement precise, let us recall the critical spaces associated with the scaling of the compressible system \eqref{MHDcomp}.  
Writing the density as $\rho = 1 + a$, the natural scaling of the equations is
$$
a_\kappa(t,x) = a(\kappa^2 t,\kappa x), \qquad
v_\kappa(t,x) = \kappa\, v(\kappa^2 t,\kappa x), \qquad
b_\kappa(t,x) = \kappa\, b(\kappa^2 t,\kappa x).
$$
A Banach space is said to be critical for \eqref{MHDcomp} if its norm is invariant under this transformation.  This leads to the classical compressible MHD critical framework
\begin{equation}\label{def:spazi critici}
a_0 \in \dot B^{\frac{d}{p}}_{p,1}(\R^d), \qquad
v_0,\, b_0 \in \dot B^{\frac{d}{p}-1}_{p,1}(\R^d).
\end{equation}

Our goal is to study the global well-posedness of \eqref{MHDcomp} for initial data in the critical spaces \eqref{def:spazi critici}, under the assumption that the bulk viscosity is sufficiently large, and to analyze the asymptotic behavior of the solutions as this parameter tends to infinity.
We remark that the bulk viscosity does not affect the natural scaling of the equations and therefore plays no role in the definition of the critical functional framework.
We prove that as $\lambda \to \infty$, the solution of the compressible system \eqref{MHDcomp} converges to the solution of the incompressible MHD equations
\begin{equation}\label{MHDincomp}
    \begin{cases}
     V_t + V\cdot \nabla V - \mu \Delta V+\nabla\Pi=B\cdot \nabla B, \\
     B_t+V\cdot \nabla B-B\cdot \nabla V-\eta\Delta B=0,\\
     \dive V=0, \qquad \dive B=0, \\
  V|_{t=0} = V_0,\qquad B|_{t=0} = B_0,\\
    \end{cases}
    \tag{I-MHD}
\end{equation}
with $V_0$ and $B_0$ denoting the Leray-Helmholtz projection of $v_0$ and $b_0$, respectively, on divergence-free vector fields.
Since $\dive b_0=0$, we simply have $B_0=b_0$. However, we keep the notation with capital letters to distinguish the incompressible framework from the compressible one. We point out that the natural critical spaces for the initial data associated with the scaling of \eqref{MHDincomp} are 
\begin{equation}
V_0,\, B_0 \in \dot B^{\frac{d}{p}-1}_{p,1}(\R^d).
\end{equation}

Before stating our main theorem, we recall several well-posedness results for \eqref{MHDincomp}. In the resistive and viscous case $(\eta>0,\ \mu>0)$, global finite-energy weak solutions and local strong solutions to \eqref{MHDincomp} in two and three dimensions were constructed in \cite{DL72}. For smooth initial data, global smoothness and uniqueness in two dimensions also follow from their analysis. In three dimensions, uniqueness of local strong solutions and several regularity criteria were obtained in \cite{ST}. Global mild solutions for small initial data were established in \cite{MYZ, WW}. This situation is somewhat reminiscent of the available results for the Navier-Stokes equations. In two dimensions, global weak solutions were obtained in \cite{K} for divergence-free initial data in $L^2$. The ideal case $\eta=0$ has been extensively studied, with essentially sharp local well-posedness results in Sobolev spaces established in \cite{FMRR, Feff}. When both viscosity and resistivity vanish $(\mu=\eta=0)$, local strong solutions for initial data in $H^s(\R^d)$, $s>d/2+1$, were proved in \cite{Sch, Secchi}.

Global smooth solutions in the ideal regime for small perturbations of constant states are available in \cite{LXZ, LZ, RWXZ, XZ}. In addition, global well-posedness for axisymmetric data in the three-dimensional inviscid or non-resistive case was shown in \cite{AHH, H, Lei}. Finally, Yudovich-type local well-posedness in the 2D ideal setting was proved in \cite{Hmidi}.\\
\\
We denote by $\cp$ and $\cq$ the standard Leray projectors. Then, our first main result reads as follows. 

\begin{restatable}{mainthm}{MainTheoremA}\label{thm:main}
    Let $d\geq2$ and $c_0>0$. Assume that the initial conditions satisfy
    \begin{equation}\label{dato iniziale main}
        v_0,b_0\in\dot B_{2,1}^{d/2-1}(\R^d), \quad
        a_0:=(\rho_0-1)\in\dot B_{2,1}^{d/2}(\R^d), \quad  c_0\leq \rho_0\leq c_0^{-1},
    \end{equation}
and that the initial data $V_0:=\cp v_0$ and $B_0:=\cp b_0 $ generates a unique global solution $V,B\in C_b(\R_+,\dot B_{2,1}^{d/2-1}(\R^d))$ of \eqref{MHDincomp}. Denote by
\begin{equation*}
    M:=\|V\|_{L^\infty(\R_+;\dot B_{2,1}^{d/2-1}) }+\|V_t,\mu \nabla^2 V\|_{L^1(\R_+;\dot B_{2,1}^{d/2-1}) }+\|B\|_{L^\infty(\R_+;\dot B_{2,1}^{d/2-1}) }+\|B_t,\eta \nabla^2 B\|_{L^1(\R_+;\dot B_{2,1}^{d/2-1}) }.
\end{equation*}
There exists a (large) universal constant $C$ such that if $\nu\geq \mu$ and
    \begin{equation*}
        Ce^{C(M+M^2)}(\|\cq u_0\|_{\dot B_{2,1}^{d/2-1}}+\|a_0\|_{\dot B_{2,1}^{d/2-1} }+\nu\|a_0\|_{\dot B_{2,1}^{d/2} } +M^2+\mu^2)\leq \sqrt{\mu\nu},
    \end{equation*}
then \eqref{MHDcomp} admit a unique global-in-time solution $(\rho,v,b) $ such that 
    \begin{equation*}
        \begin{aligned}
            &v \in C_b(\R_+;\dot B_{2,1}^{d/2-1}), \quad v_t,\nabla^2 v\in L^1(\R_+;\dot B_{2,1}^{d/2-1}), \\
            &b \in C_b(\R_+;\dot B_{2,1}^{d/2-1}), \quad b_t,\nabla^2 b\in L^1(\R_+;\dot B_{2,1}^{d/2-1}), \\
            &a:=(\rho-1)\in C(\R_+;\dot B_{2,1}^{d/2-1}\cap\dot B_{2,1}^{d/2})\cap L^2(\R_+;\dot B_{2,1}^{d/2}).
        \end{aligned}
    \end{equation*}
    In addition, the compressible components satisfy the a priori estimate
    \begin{equation*}
        \begin{aligned}
            \|\cq v\|_{L^\infty(\R_+;\dot B_{2,1}^{d/2-1})}&+\|\nu\nabla^2\cq v\|_{L^1(\R_+;\dot B_{2,1}^{d/2-1})}+\|a\|_{L^\infty(\R_+;\dot B_{2,1}^{d/2-1}) }+\nu\|a\|_{L^\infty(\R_+;\dot B_{2,1}^{d/2} )} \\
            &\lesssim e^{C(M+M^2)}(\|a_0\|_{\dot B_{2,1}^{d/2-1}}+\nu\|a_0\|_{\dot B_{2,1}^{d/2}}+\|\cq u_0\|_{\dot B_{2,1}^{d/2-1}}+M^2+\mu^2).
        \end{aligned}
    \end{equation*}
Moreover, if $a_0=0$ then the triple $(\rho,v,b)$ converges, as $\lambda\to\infty$, towards the incompressible solution $(1,V,B)$ with the following quantitative rates
    \begin{equation}\label{estconv}
        \begin{aligned}
            \sqrt{\dfrac{\nu}{\mu}} &\biggl( \| \rho-1\|_{L^\infty(\R_+; \dot B_{2,1}^{d/2} )}+\|\nabla^2\cq v\|_{L^1(\R_+; \dot B_{2,1}^{d/2-1} )}\biggr)
            +\|\cp v-V\|_{L^\infty (\R_+; \dot B_{2,1}^{d/2-1} ) }+\|b-B\|_{L^\infty (\R_+; \dot B_{2,1}^{d/2-1} ) }\\
           & +\|\cp v_t-V_t,\mu \nabla^2(\cp v-V) \|_{L^1(\R_+; \dot B_{2,1}^{d/2-1} )}
           +\|b_t-B_t,\eta \nabla^2(b-B) \|_{L^1(\R_+; \dot B_{2,1}^{d/2-1} )}\leq C\sqrt{\dfrac{\mu}{\nu}}.
        \end{aligned}
    \end{equation}
Furthermore, if $d=2$ we have an explicit bound on $M$ in terms of the norm of the initial data, which is
\begin{equation*}
M\leq C\left( \|V_0\|_{\dot B_{2,1}^0}+\|B_0\|_{\dot B_{2,1}^0} \right)\exp\biggl( C\biggl( \frac{1}{\mu^4}+\frac{1}{\mu^3\eta}+\frac{1}{\eta^3\mu}+\frac{1}{\eta^4}\biggr)\left(\|V_0\|_{L^2}^4+\|B_0\|_{L^2}^4 \right)\biggr).
\end{equation*}
\end{restatable}

The theorem is stated in a general spatial dimension $d$ under the assumption that the corresponding incompressible system admits a global strong solution. While this remains an assumption in dimension $d\geq 3$, it is automatically satisfied in dimension $d=2$, where global well-posedness of \eqref{MHDincomp} is known. Accordingly, in the three-dimensional case we assume that the initial data $V_0,B_0$ generate a global strong solution to \eqref{MHDincomp}, and we then analyze the large bulk viscosity regime for the compressible system \eqref{MHDcomp}.
Nonetheless, global well-posedness results for \eqref{MHDincomp} have been established in a variety of settings, ranging from small initial data regimes \cite{DL72, fujii, MY, ST, WD} to frameworks relying on suitable structural assumptions on the initial data \cite{HHW, LZZ, LP}. Consequently, Theorem \ref{thm:main} applies within each of these contexts.\\

The proof of Theorem \ref{thm:main} is inspired by the approach introduced in \cite{danchin2016compressiblenavierstokeslarge} for the compressible Navier–Stokes equations. The starting point is a global-in-time solution to the incompressible system, which is used as a reference flow. One then considers a local strong solution to the compressible equations and studies the evolution of the differences between the two solutions.
By decomposing the velocity through the Leray projectors, the analysis separates the compressible and divergence-free components. Energy estimates are derived for each part and then combined. When the bulk viscosity parameter is sufficiently large, these estimates can be closed, leading to uniform control of the solution and allowing the use of standard continuation criteria. In the present work, we adapt this method to the MHD equations and show that the same approach remains effective once the magnetic field is taken into account. The key fact is that the magnetic field scales in the same way as the velocity: this allows to adapt the method to this setting, preserving the structure of the energy estimates. Moreover, the divergence-free property of the magnetic field must be explicitly taken into account in the definition of the energy functionals controlling the incompressible component of the solution. By carefully combining the estimates for the divergence-free part of the velocity with those for the magnetic field, we are able to close the energy inequalities and rigorously handle the nonlinear interactions between velocity and magnetic field.\\


The $L^2$-based result of \cite{danchin2016compressiblenavierstokeslarge} was later extended to $L^p$-critical Besov spaces in \cite{Chen-Zhai:Lp}. Building on this refinement, we are able to carry out an analogous extension for the compressible MHD system, thereby obtaining global strong solutions and quantitative stability estimates in the $L^p$-critical framework as well. The corresponding statement reads as follows.
\begin{restatable}{mainthm}{MainTheoremB}\label{thm:Lp intro}
Let $d\geq2$ and $c_0>0$. Assume that the initial conditions satisfy
\begin{equation*}
    \begin{aligned}
            &a_0^\ell:=(\rho_0-1)^\ell\in \dot B_{2,1}^{d/2-1}(\R^d), \quad a_0^h:=(\rho_0-1)^h\in\dot B_{p,1}^{d/p}(\R^d), \quad c_0\leq\rho_0\leq c_0^{-1},\\
            &\cp v_0\in\dot B_{p,1}^{d/p-1}(\R^d), \quad \cq v_0^\ell\in \dot B_{2,1}^{d/2-1}(\R^d), \quad \cq v_0^h\in \dot B_{p,1}^{d/p-1}(\R^d),\quad b_0\in\dot B_{p,1}^{d/p-1}(\R^d),
    \end{aligned}
\end{equation*}
with $2\leq p< 4$ if $d=2$, and $2\leq p\leq \min\{4, 2d/(d-2)\}$ for $d\geq 3$. Assume that \eqref{MHDincomp} admits a unique global solution
\begin{equation*}
    V,B\in C_b(\R_+;\dot B_{p,1}^{d/p-1}(\R^d))\cap L^1(\R_+;\dot B_{p,1}^{d/p-1}(\R^d)),
\end{equation*}
and denote by $M$ the constant
\begin{equation*}
    M:=\|V\|_{L^\infty\dot B_{p,1}^{d/p-1}}+\mu\|V\|_{L^1\dot B_{p,1}^{d/p-1}}+\|V_t\|_{L^1\dot B_{p,1}^{d/p-1}}
    +\|B\|_{L^\infty\dot B_{p,1}^{d/p-1}}+\|B\|_{L^1\dot B_{p,1}^{d/p-1}}.
\end{equation*}
There exists a (large) universal constant $C$ such that if $\nu\geq \mu$ and
\begin{equation*}
    \begin{aligned}
        Ce^{C(M+M^2)}(&\|a_0^\ell\|_{\dot B_{2,1}^{d/2-1} }+\nu\|a_0^\ell\|_{\dot B_{2,1}^{d/2} }+\nu\|a_0^h\|_{\dot B_{p,1}^{d/p} }\\
           & +\|\cq v_0^\ell\|_{\dot B_{2,1}^{d/2-1} }+\|\cq v_0^h\|_{\dot B_{p,1}^{-1+d/p}}+M^2+\mu^2)\leq \sqrt{\mu\nu},
    \end{aligned}
\end{equation*}
then \eqref{MHDcomp} admit a unique global-in-time solution $(\rho,v,b) $ such that
\begin{equation*}
    \begin{aligned}
        \cp v\in C_b(\R_+;\dot B_{p,1}^{d/p-1})\cap L^1(\R_+;\dot B_{p,1}^{d/p+1}), \quad & b\in C_b(\R_+;\dot B_{p,1}^{d/p-1})\cap L^1(\R_+;\dot B_{p,1}^{d/p+1}), \\
        a^\ell\in C_b(\R_+;\dot B_{2,1}^{d/2-1})\cap L^1(\R_+;\dot B_{2,1}^{d/2+1}), \quad & a^h\in C_b(\R_+;\dot B_{p,1}^{d/p})\cap L^1(\R_+;\dot B_{p,1}^{d/p}),\\
        \cq v^\ell\in C_b(\R_+;\dot B_{2,1}^{d/2-1})\cap L^1(\R_+;\dot B_{2,1}^{d/2+1}), \quad &\cq v^h\in C_b(\R_+;\dot B_{p,1}^{d/p-1})\cap L^1(\R_+;\dot B_{p,1}^{d/p+1}).
    \end{aligned}
\end{equation*}
Moreover, if $a_0=0$ then the triple $(\rho,v,b)$ converges, as $\lambda\to\infty$, towards the incompressible solution $(1,V,B)$ with the following quantitative rates
    \begin{equation*}
        \begin{aligned}
           \sqrt{\frac{\nu}{\mu}} \|\rho-1\|_{L^\infty\dot B_{p,1}^{d/p} }&+\|\cp v-V\|_{L^\infty\dot B_{p,1}^{d/p-1} }+\mu\|\cp v-V\|_{L^1\dot B_{p,1}^{d/p+1}}
           +\|\cp v_t-V_t\|_{L^1\dot B_{p,1}^{d/p-1}}
           \\
           &+\|b-B\|_{L^\infty\dot B_{p,1}^{d/p-1}}+\eta\|b-B\|_{L^1\dot B_{p,1}^{d/p+1}}+\|b_t-B_t\|_{L^1\dot B_{p,1}^{d/p-1}}
           \leq C\sqrt{\frac{\mu}{\nu}}.
        \end{aligned}
    \end{equation*}
\end{restatable}

To the best of our knowledge, singular limits related to large bulk viscosity for compressible fluid models have been investigated only in a limited number of works. Besides the results of \cite{danchin2016compressiblenavierstokeslarge, Danchin-Mucha CPAM, Chen-Zhai:Lp}, further contributions have appeared only recently. We refer to \cite{QLei, Lei-Xiong, Lei-Xiong2, Liu-Zhang} for the compressible Navier–Stokes equations. In the MHD setting, the large bulk viscosity limit has been studied in \cite{WWZ}, where the authors prove the global existence of weak solutions to the isentropic compressible MHD system in the whole space, allowing for large initial data and the presence of vacuum. They also establish the global-in-time convergence, as the bulk viscosity parameter tends to infinity, towards a weak solution of the inhomogeneous incompressible MHD equations. Their approach relies on compactness arguments based on the effective viscous flux and logarithmic interpolation inequalities of Desjardins type.

Our work is of a different nature and addresses the large bulk viscosity limit in a strong solution framework. More precisely, we focus on global strong solutions in critical Besov spaces and on quantitative stability estimates around a given incompressible MHD flow, under the assumption that the latter exists globally in time.\\

Although the large bulk viscosity regime may at first glance resemble a compressible-to-incompressible limit, the nature of our result is fundamentally different. Rather than describing a genuine incompressible limit, our theorem should be interpreted as a stability statement: incompressible flows arise as stable reference dynamics within the compressible system when the bulk viscosity is sufficiently large. 
In other words, solutions of \eqref{MHDincomp} persist, in a quantitative sense, as approximate solutions to the compressible system \eqref{MHDcomp}, and the compressible flow remains close to its incompressible counterpart for all times. For completeness, we mention that the incompressible limit in the low Mach number regime has been investigated in several works; see \cite{GLX, OL, HW, JJL-CMP, JJL-SIMA, KM, LiS, PZ} and the references therein. 
This stability perspective is particularly powerful when the incompressible dynamics exhibits additional qualitative or topological features that require strong control of the solution. In the present work, we exploit this mechanism to construct solutions of \eqref{MHDcomp} that exhibit magnetic reconnection. 

Magnetic reconnection is the process by which the topology of magnetic field lines in a conducting plasma changes, allowing stored magnetic energy to be rapidly transferred to other forms of energy. It is a ubiquitous phenomenon in space and laboratory plasmas and plays a central role in a wide range of dynamical processes, particularly in solar and astrophysical contexts (see \cite{MagnPhys, Priest}). Despite its fundamental importance, providing a rigorous description of reconnection within the framework of the magnetohydrodynamic equations remains a major challenge, and rigorous constructions have been obtained only recently in \cite{CCL, CL2, LP} for the incompressible system. In this work, we turn to the compressible setting and we aim to construct explicit solutions of \eqref{MHDcomp} that exhibit reconnection. This will follow as an application of our main result, Theorem \ref{thm:main}, with the analysis developed in \cite{LP}. Specifically, starting from incompressible MHD flows displaying this phenomenon \cite{LP}, the quantitative stability provided by Theorem \ref{thm:main} allows us to transfer the reconnection behavior to the compressible case, thereby obtaining compressible solutions with the same change of magnetic topology. The result is the following.
\begin{restatable}{mainthm}{MainTheoremC}
Given any constants $M,T>0$ and $0<c_0<1$, there exist smooth initial data $(\rho_0,u_0,b_0)$ on $\R^3$ such that 
\begin{equation*}
\|u_0\|_{\dot{B}^{\frac12}_{2,1}}\simeq \|b_0\|_{\dot{B}^{\frac12}_{2,1}}\simeq M, \qquad  a_0:=(\rho_0-1)\in\dot B_{2,1}^{3/2}(\R^3), \qquad  c_0\leq \rho_0\leq c_0^{-1},
\end{equation*}
for which \eqref{MHDcomp} with $\eta=\mu$ admits a unique global smooth solution $(\rho,u,b)$. Moreover, the magnetic lines at time $t=0$ are not topologically equivalent to those at time $t=T$, meaning that there exists no homeomorphism of $\R^3$ mapping the integral curves of $b(0,\cdot)$ onto those of $b(T,\cdot)$.
\end{restatable}
A more precise notion of magnetic reconnection, together with a detailed description of the construction, is postponed to Section \ref{sec:magnetic reconnection}.
To the best of our knowledge, this provides the first rigorous analytical example of magnetic reconnection within the framework of compressible magnetohydrodynamics.\\

\subsection*{Outline of the paper.}
In Section \ref{besov} we introduce the notations and we recall several background results on Besov spaces, and Stokes and MHD equations. 
Section \ref{proof} is devoted to the proof of Theorem \ref{thm:main}. 
In Section \ref{lp framework} we extend this result from the $L^2$-based Besov setting to a more general $L^p$ framework. 
Finally, in Section \ref{sec:magnetic reconnection} we apply our theorem to construct explicit examples of magnetic reconnection for the compressible system \eqref{MHDcomp}.

\section{Notations and preliminaries}\label{besov}
In this section, we introduce the notations that will be used throughout the article and we recall some basic results concerning Littlewood-Paley theory from \cite{nonlinearPDE}.

\subsection{Besov spaces}
We use the standard notation $\mathcal{S}(\mathbb{R}^d)$ to denote the space of Schwartz functions on $\R^d$, and with $\mathcal{S}'(\mathbb{R}^d)$ denote its dual, which is the space of tempered distributions. Moreover, we denote by $L^p(\R^d)$ the standard Lebesgue spaces and with $\|\cdot\|_{L^p}$ their norm. Let $\chi$ be a smooth radial non-increasing function  supported in the ball $B(0,\frac{4}{3})\subset\mathbb{R}^d$, which is identically $1$ on $B(0,\frac{3}{4})$. Then, we define the function $\varphi$ as 
$$
\varphi(\xi)=\chi(\xi/2)-\chi(\xi).
$$ 
Note that the support of $\varphi$ is the annullus $\left\{\xi\in \mathbb{R}^d: \frac{3}{4}\leq |\xi|\leq\frac{8}{3} \right\}$ and it holds that 
\begin{equation*}
    \sum_{j\in\mathbb{Z}}\varphi(2^{-j}\xi)=1 \quad \text{for all } \xi \in \mathbb{R}^d \setminus \{0\}.
\end{equation*}
We denote with $\mathcal{F}$ the Fourier transform operator and we define the functions
$$
h\coloneqq\mathcal{F}^{-1} \varphi, \qquad
\tilde h \coloneqq\mathcal{F}^{-1} \chi.
$$
Then, the homogeneous dyadic blocks $\dot\Delta_j$  and the low frequency cut-off are defined on tempered distribution by 
\begin{align*}
    \dot\Delta_j f&\coloneqq\varphi(2^{-j}D)f\coloneqq\mathcal{F}^{-1}(\varphi(2^{-j}\cdot)\mathcal{F}f)=2^{jd}h(2^j\cdot)* f,\\
    \dot S_j f&\coloneqq\chi(2^{-j}D)f\coloneqq\mathcal{F}^{-1}(\chi(2^{-j}\cdot)\mathcal{F}f)=2^{jd}\tilde h(2^j\cdot)* f,
\end{align*}
where $*$ denotes the convolution operator. Notice that it holds
$$
\dot\Delta_j=\dot S_{j+1}-\dot S_j.
$$
In the homogeneous setting, the Littlewood-Paley decomposition
\begin{equation}\label{LittPaley decomp}
    f=\sum_{j\in\mathbb{Z}}\dot \Delta_j f \quad \text{in } \mathcal{S}'(\mathbb{R}^d),
\end{equation}
may fail to converge for arbitrary tempered distributions. 
To ensure that the decomposition is well-defined, we restrict our attention to those tempered distribution $f$ satisfying the condition
\begin{equation}\label{cond on tempdis}
    \lim_{j\to-\infty}\|\dot S_j f\|_{L^\infty}=0.
\end{equation}
\begin{definition}\label{def:besov}
Let $s\in\mathbb{R},$ and $1\leq p,r\leq\infty.$ The homogeneous Besov space $\dot B_{p,r}^s(\mathbb{R}^d) $ consist of all the distributions $f\in \mathcal{S}'$ satisfying \eqref{cond on tempdis} and 
    \begin{equation}\label{def:norma besov}
        \|f\|_{\dot B_{p,r}^s }\coloneqq \left(\sum_{j\in\mathbb{Z}} 2^{rjs}\|\dot\Delta_j f\|_{L^p}^r\right)^{1/r} <\infty.
    \end{equation}
In particular, for $s< d/p$ or $s=d/p$ and $r=1$, the space $\dot B_{p,r}^s(\R^d)$ is a Banach space.
\end{definition}
\begin{remark}
As pointed out in \cite[Pag.63]{nonlinearPDE}, a distribution $f$ satisfies \eqref{def:besov} if and only if there exists some constant $C>0$ and some nonnegative sequence $\{ c_j\}_{j\in\Z}$ such that
\begin{equation}
    \forall j\in\Z,\,\|\dot\Delta_j f\|_{L^p}\leq Cc_j2^{-js}\,\quad\mbox{and}\quad \|(c_j)\|_{\ell^r}=1.
\end{equation}
\end{remark}

In this work we will focus mainly on the case $r=1$, in which the norm becomes
\begin{equation*}
    \|f\|_{\dot B_{p,1}^s }=\sum_{j\in\mathbb{Z}} 2^{js}\|\dot\Delta_j f\|_{L^p}<\infty.
\end{equation*}

We now collect some basic properties of homogeneous Besov spaces. First of all, we recall the Bernstein inequalities.
\begin{lemma}[Bernstein inequalities]\label{lem:bernstein}
Let $\mathcal{C}$ be an annulus and $\mathcal{B}$ a ball. Then, there exists a constant $C>0$ such that for any non-negative integer $k$, any couple $(p,q)\in[1,\infty]^2$ with $q\geq p\geq 1$, and any function $f\in L^p$, we have
\begin{align}
    \mathrm{supp}&\,\hat f\subset \lambda \mathcal{B}\Longrightarrow\|\nabla^k f\|_{L^p}\leq C^{k+1}\lambda^{k+d\left(\frac1p-\frac1q\right)}\|f\|_{L^p},\\
    \mathrm{supp}&\,\hat f\subset \lambda \mathcal{C}\Longrightarrow C^{-k-1}\lambda^k\|f\|_{L^p}\leq \|\nabla^k f\|_{L^p}\leq C^{k+1}\lambda^k\|f\|_{L^p}.
\end{align}
In particular, we have that
\begin{itemize}
\item for all $1\leq p\leq q\leq\infty$ and $k\in\mathbb{N},  $
\begin{equation}\label{Bernstein}
    \|\dot\Delta_j\nabla^k u\|_{L^q(\mathbb{R^d}) } \leq C2^{j(k+d(\frac{1}{p}-\frac{1}{q})) }\|\dot\Delta_j u\|_{L^p(\mathbb{R^d}) }.
\end{equation}
\item for all $1\leq p\leq \infty$, we have 
\begin{equation}\label{Reverse Bernstein}
    \|\dot\Delta_j u\|_{L^p(\mathbb{R^d}) }\leq C2^{-j}\|\dot\Delta_j \nabla u\|_{L^p(\mathbb{R^d}) }.
\end{equation}
\end{itemize}
Morevore, we have that dor any $s\in\R$ and $p,r\in[1,\infty]$ it holds that
\begin{equation}\label{equivalenza}
\|f\|_{B^{s}_{p,r}}\simeq\|\nabla f\|_{B^{s-1}_{p,r}}.
\end{equation}
\end{lemma}
From the above inequalities, one can prove the following proposition, see in particular Proposition 2.20 and Proposition 2.39 in \cite{nonlinearPDE}.
\begin{proposition}[Embeddings]\label{prop:embeddings}
Let $1\leq p_1< p_2\leq \infty$ and $1\leq r_1< r_2\leq \infty$. Then, for any real number $s$, the space $\dot B_{p_1,r_1}^s(\R^d) $ is continuosly embedded in $\dot B_{p_2,r_2}^{s-d\left(\frac{1}{p_1}-\frac{1}{p_2}\right) }(\R^d)$. In particular, one has that:
\begin{equation}\label{embedd B d/2 in L inft}
    \dot B_{2,1}^{d/2}(\R^d)\hookrightarrow L^\infty(\R^d).
\end{equation}
\end{proposition}

\begin{proposition}[Interpolation]
For any $s_1,s_2\in \R$ with $s_1<s_2$ and any $\theta\in (0,1)$, then we have that for any $p,r\in[1,+\infty]$
\begin{equation}
\|f\|_{\dot B_{p,r}^{\theta s_1+(1-\theta) s_2}}\leq \|f\|^\theta_{\dot B^{s_1}_{p,r}}\|f\|^{1-\theta}_{\dot B^{s_2}_{p,r}}.
\end{equation}
In particular, by using \eqref{equivalenza} we have that
\begin{equation}\label{interpolation}
\|f\|_{\dot B_{2,1}^{d/2} }\leq C  \|f\|_{\dot B_{2,1}^{d/2-1} }^{\frac{1}{2}}\|\nabla^2 f\|_{\dot B_{2,1}^{d/2-1} }^{\frac{1}{2}}.
\end{equation}
\end{proposition}


To deal with nonlinear terms, we make use of the {\em Bony decomposition}, see \cite[Definition 2.4]{nonlinearPDE}, which offers a convenient framework for handling products in Besov spaces. \begin{definition}\label{def:Bony decomp}
The homogeneous paraproduct of $h$ by $g$ is defined as follows
\begin{equation}
    \begin{aligned}\label{Bony decomp1}
        \dot T_g h:= \sum_{j\in\Z } \dot S_{j-1}g\dot\Delta_j h.
    \end{aligned}
\end{equation}
The homogeneous remainder of $g$ and $h$ is defined as
\begin{equation}\label{Bony decomp2}
\dot R(g,h):= \sum_{|k-j|\leq 1 } \dot \Delta_k g \dot\Delta_jh
\end{equation}
When all these terms make sense in $\mathcal S'(\R^d)$, the {\em Bony decomposition} of the product $gh$ is then given by
\begin{equation*}
    gh=\dot T_g h+\dot T_h g+\dot R(g,h).
\end{equation*}
\end{definition}
To make use of Bony’s decomposition we rely on the following continuity estimates for the paraproduct and remainder operators.
\begin{lemma}\label{continuity bony decomp}
    Let $s,\ s_1,\ s_2\in\R$, $1\leq p_1,\ p_2,\ r_1,\ r_2\leq\infty,$ $\frac{1}{p}=\frac{1}{p_1}+\frac{1}{p_2},$ $\frac{1}{r}=\frac{1}{r_1}+\frac{1}{r_2}$ and $\tau<0.$ Then we have
    \begin{equation*}
        \begin{aligned}
            \|\dot T_g h\|_{\dot B_{p,r}^s }&\lesssim \|g\|_{L^{p_1} }\|h\|_{\dot B_{p_2,r}^s },\\
            \|\dot T_g h\|_{\dot B_{p,r}^{s+\tau} }&\lesssim \|g\|_{\dot B_{p_1,\infty}^\tau }\|h\|_{\dot B_{p_2,r}^s }, \\
            \|\dot R(g,h)\|_{\dot B_{p,r}^{s_1+s_2} }&\lesssim \|g\|_{\dot B_{p_1,r_1}^{s_1} }\|h\|_{\dot B_{p_2,r_2}^{s_2} }, \qquad \text{for } s_1+s_2>0, \\
            \|\dot R(g,h)\|_{\dot B_{p,\infty}^0 }&\lesssim \|g\|_{\dot B_{p_1,r_1}^{s_1} }\|h\|_{\dot B_{p_2,r_2}^{s_2} }, \qquad \text{for } s_1+s_2=0.
        \end{aligned}
    \end{equation*}
\end{lemma}
As a consequence, we can state the following lemma, which will be used frequently in the paper and provides an estimate for the Besov norm of the product of two functions.
\begin{lemma}\label{lemmastimaprodLP}
    Let $1\leq p,q\leq\infty$, $s_1\leq \frac{d}{q}$, $s_2\leq d\text{ min}\{\frac{1}{p}, \frac{1}{q} \}$, and $s_1+s_2> d \text{ max}\{0,\frac{1}{p}+\frac{1}{q}-1 \}$.
    Then we have, for $g \in \dot B^{s_1}_{q,1}(\R^d)$ and  $h\in\dot B^{s_2}_{p,1}(\R^d)$,
    \begin{equation*}
        \|gh\|_{\dot B_{p,1}^{s_1+s_2-\frac{d}{p} } }\lesssim \|g\|_{\dot B_{q,1}^{s_1} }\|h\|_{\dot B_{p,1}^{s_2}}
    \end{equation*}
\end{lemma}
In particular, for the case $p=2$, we have the following.
\begin{lemma}\label{lemmastimaprod}
Let $g \in \dot B^{s_1}_{2,1}(\R^d)$ and  $h\in\dot B^{s_2}_{2,1}(\R^d)$ for some couple $(s_1,s_2)$ satisfying 
$$
s_1, s_2\leq d/2\ \hbox{ and }\ s_1+s_2>0.
$$
Then $gh\in\dot B^{s_1+s_2-d/2}_{2,1}(\R^d),$ and we have
\begin{equation}
\|gh\|_{ \dot B^{s_1+s_2-d/2}_{2,1}} \leq C \|g\|_{\dot B^{s_1}_{2,1}}\|h\|_{\dot B^{s_2}_{2,1}}.
\end{equation}
\end{lemma}
We use the notation $[A,B]=AB-BA$ to denote the commutator between two operators $A,B$. We now recall a set of commutator estimates, stated only in the form needed for our analysis. The first result is a consequence of \cite[Lemma 2.100]{nonlinearPDE} together with the embeddings $L^\infty\hookrightarrow\dot B_{\infty,\infty}^0\hookrightarrow\dot B_{2,\infty}^{d/2}$ from Proposition \ref{prop:embeddings}.
\begin{lemma}[Commutator 1]\label{lemma commut 2.100}
    Let $\sigma\in\R$, and $1\leq p\leq p_1\leq\infty$. Assume that $\sigma<1+\frac{d}{p_1}$ and
    \begin{equation*}
        \sigma>-d\min\left\{\frac{1}{p_1},\frac{1}{p'} \right\} \quad \text{or} \quad \sigma>-1-d\min\left\{\frac{1}{p_1},\frac{1}{p'} \right\} \quad \text{if} \quad \dive g=0.
    \end{equation*}
    Define $R_j:=[g\cdot\nabla,\dot\Delta_j]f $ (or $R_j:=\dive([g,\dot\Delta_j]f) $, if $\dive g=0).$ There exists a constant $C$, depending continuously on $p,p',\sigma$ and $d$, such that
    \begin{equation}\label{commutatorestimL2}
        2^{j\sigma}\|[g\cdot\nabla,\dot\Delta_j]f\|_{L^2}\leq C c_j\|\nabla g\|_{L^\infty}\|f\|_{\dot B_{2,1}^{d/2-1}} \quad \text{with} \quad \sum_{j\in\Z}c_j\leq1.
    \end{equation}
\end{lemma}
On several occasions, it will be convenient to decompose tempered distributions $f$ into their {\em low-} and {\em high-frequency} components, defined by
\begin{equation}\label{def:split frequenze}
    f^\ell\coloneqq \sum_{2^k\nu\leq1 }\dot\Delta_k f \quad \text{and }\quad  f^h\coloneqq \sum_{2^k\nu>1 }\dot\Delta_k f.
\end{equation}
The next estimates, which will be used repeatedly, follow from \cite[Lemma~2.16]{Zhai-Li:GlobLargsol3dCNS}.
\begin{lemma}[Commutator 2]\label{commut est Lpsetting}
    Let $2\leq p\leq \min\{ 4, 2d(d-2)\} $ for $d>2$ and $2\leq p<4$ for $d=2$. Assume $A(D)$ a zero-order Fourier multiplier. For $v^\ell\in \dot B_{2,1}^{d/2-1}(\R^d) $, $v^h\in \dot B_{p,1}^{d/p-1}(\R^d)$ and $\nabla u\in \dot B_{p,1}^{d/p}(\R^d) $, we have
    \begin{equation*}
        \begin{aligned}
            &\sum_{j\leq j_0} 2^{(d/2-1)j }\|\dot \Delta_j([A(D),u\cdot\nabla]v)\|_{L^2}\leq C\left( \|\nabla u^\ell\|_{\dot B_{2,1}^{d/2}}+\|\nabla u^h\|_{\dot B_{p,1}^{d/p}} \right)\left( \|v^\ell\|_{\dot B_{2,1}^{d/2-1}}+\|v^h\|_{\dot B_{p,1}^{d/p-1}} \right),\\
            &\sum_{j\leq j_0} 2^{(d/2-1)j }\|\dot \Delta_j([A(D),u\cdot\nabla]v)\|_{L^2}\leq C \|\nabla u\|_{\dot B_{p,1}^{d/p}}\left( \|v^\ell\|_{\dot B_{2,1}^{d/2-1}}+\|v^h\|_{\dot B_{p,1}^{d/p-1}}\right), \quad \text{if } \dive u=0,
        \end{aligned}
    \end{equation*}
    for a constant dependent on $j_0$.
\end{lemma}
We recall that zero-order Fourier multipliers are operators whose symbols are bounded in frequency, and hence are bounded on $L^p$ and on Besov spaces $\dot B^s_{p,q}$ for all $s\in\mathbb{R}$. We also state the following commutator lemma, see \cite[Lemma 6.1]{DH}.
\begin{lemma}[Commutator 3]\label{lem:commutator 3}
Let $A(D)$ be a zero-oreder Fourier multiplier. Let $j_0\in\Z$, $\tau\in\R$, $1\leq p_1,p_2\leq \infty$ and $1/p=1/p_1+1/p_2$. Then, we have
\begin{align}
    \|[\dot S_{j_0}A(D), T_f]g\|_{\dot B^{\tau+s}_{p,1}}&\leq C\|\nabla f\|_{\dot B^{s-1}_{p,1}}\|g\|_{\dot B^{\tau}_{p_2,\infty}},\quad \mbox{ if }s<1,\\
    \|[\dot S_{j_0}A(D), T_f]g\|_{\dot B^{\tau+1}_{p,1}}&\leq C\|\nabla f\|_{L^{p_1}}\|g\|_{\dot B^{\tau}_{p_2,1}},\quad \mbox{ if }s=1,
\end{align}
\end{lemma}
As a final ingredient from the theory of Besov spaces, we recall the following composition estimate.
\begin{lemma}[Composition]\label{composit est Lp} 
Let $G$ with $G(0)=0$ be a smooth function defined on an open interval $I$ of $\R$ containing 0. Then the following estimates
\begin{equation*}
    \|G(f)\|_{\dot B_{p,1}^s }\lesssim\|f\|_{\dot B_{p,1}^s } \quad \text{and} \quad \|G(f)\|_{\tilde{L}_T^q(\dot B_{p,1}^s) }\lesssim \|f\|_{\tilde{L}_T^q(\dot B_{p,1}^s) }
\end{equation*}
    hold true for $s>0,$ $1\leq p, \ q\leq\infty$ and $f$ valued in a bounded interval $J\subset I. $
\end{lemma}

We conclude this section with some furhter notations.
Given a Besov space $\dot B^s_{p,q}(\R^d)$, we denote by $L^r((0,T);\dot B^s_{p,q}(\R^d))$ the space of all measurable functions $f:[0,T]\to \dot B^s_{p,q}(\R^d)$ such that
$$
\|f\|_{L^r\dot B^s_{p,q}}:=\left(  \int_0^T\|f(t,\cdot)\|^r_{\dot B^s_{p,q}} \de t\right)^{\frac{1}{r}}<\infty,
$$
for all $1\leq p<\infty$, while for $p=\infty$ we define
$$
\|f\|_{L^\infty\dot B^s_{p,q}}:=\esssup_{t\in [0,T]}\|f(t,\cdot)\|_{\dot B^s_{p,q}}<\infty.
$$
We occasionally write $\|\cdot\|_{L^r_T\dot B^s_{p,q}}$ for the norm in $L^r((0,T);\dot B^s_{p,q}(\R^d))$.
The space of continuous functions on $[0,T]$ with values in $\dot B^s_{p,q}(\R^d)$ is denoted by $C([0,T];\dot B^s_{p,q}(\R^d))$ and is also endowed with the supremum norm $\|\cdot\|_{L^\infty\dot B^s_{p,q}}$, and we use the usual notation $C_b(\R_+;\dot B^{s}_{p,q}(\R^d))$ for the space 
$$
C(\R_+;\dot B^{s}_{p,q}(\R^d))\cap L^\infty(\R_+;\dot B^{s}_{p,q}(\R^d)).
$$

Moreover, for $1\leq p\leq \infty$ and $1\leq r\leq\infty$, we denote by $\tilde{L}^r_t\dot B^s_{p,1}$, the Chemin-Lerner type Besov space, endowed with the norm
\begin{equation}\label{def:spazi Chemin-Lerner}
\|f\|_{\tilde{L}_T^r(\dot B_{p,1}^s) }:=\sum_{j\in\Z}2^{js}\|\dot\Delta_j f\|_{L^r L^p}.
\end{equation}
We remark that, by Minkowski's inequality, it follows that
\begin{equation*}
\|f\|_{L^r_T(\dot B_{p,1}^s) } \leq \|f\|_{\tilde{L}_T^r(\dot B_{p,1}^s)}.
\end{equation*}

Lastly, we recall that we use the standard notation for the Leray projectors $\cp$ and $\cq$, which are defined as
\begin{equation*}
    \cq\coloneqq-(-\Delta)^{-1}\nabla\dive \quad \text{and } \cp\coloneqq \text{Id}+(-\Delta)^{-1}\nabla\dive.
\end{equation*}
Since $\cp$ and $\cq$ are smooth, homogeneous Fourier multipliers of degree zero, they act continuously on $\dot B_{2,1}^s(\R^d) $ for any $s\leq d/2$.

\subsection{Preliminaries on Stokes and MHD systems} This subsection is devoted to recalling several analytical results for the Stokes, incompressible MHD, and compressible MHD systems, with particular emphasis on well-posedness and energy estimates in Besov spaces.
First of all, we state a maximal regularity result in Besov spaces for the Stokes system
\begin{equation}\label{Stokes system}
   \begin{cases}
        u_t-\nu\Delta u+\nabla P=f, \\
        \dive u =g, \\
        u|_{t=0}=u_0.
  \end{cases}
\end{equation}

\begin{theorem}\cite[Theorem 4.1.1]{Danchin-Mucha regStokes} \label{thm max reg Stokes}
    Let $p\in(1,\infty)$ and $1/p-1<s<1/p.$ Let $f\in L^1((0,T);\dot B_{p,1}^s(\R^d))$,$ g\in C([0,T];\dot B_{p,1}^{s-1}(\R^d))$ with $\nabla g\in L^1((0,T);\dot B_{p,1}^s(\R^d)) $ and $u_0\in\dot B_{p,1}^s(\R^d). $ Assume in addition that 
    \begin{equation*}
        g_t=\dive B+A, \quad \text{with} \quad \mathrm{supp\,} A(t,\cdot)\subset \overline{B(0,\lambda)} \quad \text{and} \quad \int_{\R^d} A(t,x)\de x=0
    \end{equation*}
    for some $\lambda>0$ and $A,\ B\in L^1((0,T);\dot B_{p,1}^s(\R^d))$ and that the compatibility condition $\dive u_0=g|_{t=0} $ on $\R^d$ (in the distributional meaning) is satisfied. \\
    Then, the system \eqref{Stokes system} has a unique solution $(u,\nabla P)$ with
    \begin{equation*}
        u\in C([0,T);\dot B_{p,1}^s(\R^d)) \quad \text{and} \quad \partial_t u,\nabla^2 u, \nabla P\in L^1((0,T);\dot B_{p,1}^s(\R^d))
    \end{equation*}
    and the following estimate is valid
    \begin{equation}\label{max regul Stokes system}
    \begin{aligned}
        \|u\|_{L^\infty(0,T;\dot B_{p,1}^s(\R^d)) }&+\|u_t,\nu\nabla^2 u,\nabla P\|_{L^1(0,T;\dot B_{p,1}^s(\R^d))}\\
        &\leq C\left( \|f,\nu\nabla g,B\|_{L^1(0,T;\dot B_{p,1}^s(\R^d))}+\lambda \|A\|_{L^1(0,T;\dot B_{p,1}^s(\R^d))}+\|u_0\|_{\dot B_{p,1}^s(\R^d)} \right),
    \end{aligned}
    \end{equation}
    where $C$ is an absolute constant with no depedence on $\nu,$ $T$ and $\lambda.$
\end{theorem}

We now give a proof of a classical real interpolation result
\begin{proposition}\label{prop:real interpolation}
Let $f\in L^2\cap \dot H^1(\R^d)$. Then $f\in\dot B^\frac12_{2,1}$ and the following inequality holds
\begin{equation}
    \|f\|_{B^\frac12_{2,1}}\leq 4 \|f\|_{L^2}^\frac12\|f\|_{\dot H^1}^\frac12
\end{equation}
\end{proposition}
\begin{proof}
We denote by $K(s,f;L^2,\dot H^1)$ the Peetre K-functional: $\forall s\geq 0$ and $\forall f\in L^2+\dot H^1$ it is defined as
\begin{equation}
    K(s,f;L^2,\dot H^1):=\inf_{f=f_0+f_1}\|f_0\|_{L^2}+s\|f_1\|_{\dot H^1}.
\end{equation}
We recall that the Besov space $\dot B^\frac12_{2,1}(\R^d)$ can be defined via real interpolation, see \cite[Theorem 6.3.1]{BL}, as
\begin{equation*}
    (L^2(\R^d),\dot H^1(\R^d))_{1/2,1}=\dot B_{2,1}^{1/2}(\R^d),
\end{equation*}
and the norm can be computed via the Peetre K-functional as
\begin{equation}
    \|f\|_{\dot B^\frac12_{2,1}}=\int_0^\infty s^{-\frac12}K(s,f;L^2,\dot H^1)\frac{\de s}{s}.
\end{equation}
Since $f\in L^2\cap \dot H^1(\R^d)$, we have that
\begin{equation}
    K(s,f;L^2,\dot H^1)\leq \min\{\|f\|_{L^2},s\|f\|_{\dot H^1}\}.
\end{equation}
Then, by defining $\bar s=\|f\|_{L^2}/\|f\|_{\dot H^1}$, we obtain that
\begin{align*}
\|f\|_{\dot B^\frac12_{2,1}}&=\int_0^\infty s^{-\frac12}K(s,f;L^2,\dot H^1)\frac{\de s}{s}\\
&\leq \int_0^{\bar s} s^{-\frac12}s\|f\|_{\dot H^1}\frac{\de s}{s}+\int_{\bar s}^{\infty} s^{-\frac12}\|f\|_{L^2}\frac{\de s}{s}\\
&=2\|f\|_{\dot H^1}\sqrt{\bar s}+2\|f\|_{L^2}\frac{1}{\sqrt{\bar s}}=4\|f\|_{L^2}^\frac12\|f\|_{\dot H^1}^\frac12,
\end{align*}
and this concludes the proof.
\end{proof}

Now, we recall the well-posedness result of the incompressible MHD equations in the Besov space $\dot B^0_{2,1}$ from \cite{MY}. Here we complement this result with additional energy estimates, which will be used later to obtain a quantitative bound on the constant $M$ in Theorem \ref{thm:main}. 
\begin{lemma}\label{lem:M}
Let $V_0,B_0\in \dot B^0_{2,1}(\R^2)$ be divergence-free vector fields. Then, there exists a unique solution $(V,B)$ of \eqref{MHDincomp} such that
\begin{align*}
&V,B\in C_b(\R^+;\dot B^0_{2,1}(\R^2))\cap L^2(\R_{+};\dot B_{2,1}^1(\R^2)),\\
& V_t,\, B_t, \, \nabla^2 V,\, \nabla^2 B\in L^1(\R_{+};\dot B_{2,1}^0(\R^2)).
\end{align*}
Moreover, for any $T\geq 0$ the following estimate holds
\begin{equation}\label{stima espli M dim2}
   \begin{aligned}
       \|V\|_{L^\infty\dot B^0_{2,1}}&+\|V_t,\mu\nabla^2 V\|_{L^1\dot B^0_{2,1}}+\|B\|_{L^\infty\dot B^0_{2,1}}+\|B_t,\eta\nabla^2 B\|_{L^1\dot B^0_{2,1}}\\
         &\leq C\left( \|V_0\|_{\dot B_{2,1}^0}+\|B_0\|_{\dot B_{2,1}^0} \right)\exp\biggl( C\biggl( \frac{1}{\mu^4}+\frac{1}{\mu^3\eta}+\frac{1}{\eta^3\mu}+\frac{1}{\eta^4}\biggr)\left(\|V_0\|_{L^2}^4+\|B_0\|_{L^2}^4 \right)\biggr).
   \end{aligned}
\end{equation}
\end{lemma}

\begin{proof}
First of all, in the two-dimensional case we have the embedding
$$\dot B_{2,1}^0(\R^2)\hookrightarrow L^2(\R^2). $$ 
Thus, the initial data $V_0,B_0\in L^2(\R^2)$ and from \cite{DL72} one has that there exists a unique global solution to \eqref{MHDincomp} satisfying
\begin{equation}\label{reg v B}
    (V,B)\in C_b([0,\infty);L^2(\R^2))\cap L^2([0,\infty); \dot H^1(\R^2)).
\end{equation}
In \cite{MY}, it is proved that the additional regularity is propagated and then the solution belongs to
\begin{equation}
    (V,B)\in C([0,\infty]),\dot B_{2,1}^0(\R^2)).
\end{equation}
By Proposition \ref{prop:real interpolation} we have that
\begin{equation*}
    \|V\|_{(L^2(\R^2),\dot H^1(\R^2))_{1/2,1} }\lesssim \|V\|_{L^2}^{1/2}\|V\|_{\dot H^1}^{1/2},
\end{equation*}
and then, the following bound holds
\begin{equation}\label{interpL4B1/2}
    \begin{aligned}
        \|V\|_{L^4\dot B_{2,1}^{1/2}}^4&
        \lesssim \int_0^{+\infty} \left(\|V\|_{L^2}^{1/2}\|V\|_{\dot H^1}^{1/2} \right)^4\de t\lesssim \|V\|_{L^\infty L^2}^2\|V\|_{L^2 \dot H^1}^2.
    \end{aligned}
\end{equation}
Since $V$ and $B$ share the same regularity, the analogous bound holds for $B$, and therefore
\begin{equation}
   ( V,B)\in L^4([0,\infty);\dot B_{2,1}^{1/2} (\R^2)).
\end{equation}
Then, we use the divergence-free assumption and Lemma \ref{lemmastimaprod} to bound the product
\begin{equation*}
    \begin{aligned}
        \int_0^{+\infty} \|V\cdot\nabla B\|_{\dot B_{2,1}^{-1} }^2\de t=&\int_0^{+\infty} \|V\otimes B\|_{\dot B_{2,1}^{0}}^2\de t 
        \lesssim \int_0^{+\infty} \|V\|_{\dot B_{2,1}^{1/2}}^2\|B\|_{\dot B_{2,1}^{1/2}}^2 \de t
        \lesssim \|V\|_{L^4\dot B_{2,1}^{1/2}}\|B\|_{L^4\dot B_{2,1}^{1/2}}.
    \end{aligned}
\end{equation*}
A similar argument yields that 
\begin{equation*}
    V\cdot\nabla V,\ B\cdot\nabla B,\ V\cdot\nabla B,\ B\cdot\nabla V \in L^2(\R_{+};\dot B_{2,1}^{-1}(\R^2)).
\end{equation*}
Using the equation \ref{MHDincomp}, one can also deduce that 
\begin{equation}\label{reg V B da sist}
   V,B\in L^2(\R_{+};\dot B_{2,1}^1(\R^2)). 
\end{equation}
Finally, applying Lemma \ref{lemmastimaprod} again, we conclude that $$V\cdot\nabla V,\ B\cdot\nabla B,\ V\cdot\nabla B,\ B\cdot\nabla V\in L^1(\R_{+};\dot B_{2,1}^0),$$ and thus
\begin{equation}
\begin{aligned}
     &V_t, \ \nabla^2 V\in L^1(\R_{+};\dot B_{2,1}^0(\R^2)), \\
     &B_t,\ \nabla^2 B\in L^1(\R_{+};\dot B_{2,1}^0(\R^2)).
\end{aligned}    
\end{equation}
To summarize, in dimension two, for incompressible initial data $V_0, B_0\in \dot B_{2,1}^0(\R^2)$, there exists a unique global smooth solution belonging to the regularity class
\begin{equation}
    (V,B)\in C_{b}(\R_{+};\dot B_{2,1}^{0}(\R^2) )\cap L^1(\R_{+};\dot B_{2,1}^{2}(\R^2)).
\end{equation}
We now prove \eqref{stima espli M dim2}. 
From the endpoint maximal regularity estimates of the Stokes system in homogeneous Besov spaces (see Theorem \ref{thm max reg Stokes}),
we have that for all $T\geq0$
\begin{align}
    &\|V\|_{L^\infty\dot B^0_{2,1}}+\|V_t,\mu\nabla^2 V\|_{L^1\dot B^0_{2,1} }\leq 
    C\left(\|V\cdot\nabla V\|_{L^1\dot B^0_{2,1}}+\|B\cdot\nabla B\|_{L^1\dot B^0_{2,1}}+\|V_0\|_{\dot B^0_{2,1}} \right),\label{maxregV}\\
    &\|B\|_{L^\infty\dot B^0_{2,1}}+\|B_t,\eta\nabla^2 B\|_{L^1\dot B^0_{2,1}}\leq 
    C\left(\|V\cdot\nabla B\|_{L^1\dot B^0_{2,1}}+\|B\cdot\nabla V\|_{L^1\dot B^0_{2,1}}+\|B_0\|_{\dot B^0_{2,1}} \right)\label{maxregB},
\end{align}
where the temporal norms are computed over the interval $(0,T)$.
In order to estimate the nonlinear terms in \eqref{maxregV} and \eqref{maxregB}, we use Lemma \ref{lemmastimaprod}, together with the interpolation inequality
\begin{equation}
    \|Z\|_{\dot B_{2,1}^{1/2}}\leq C \|Z\|_{\dot B_{2,1}^{-1}}^{1/4}\|\nabla Z\|_{\dot B_{2,1}^{0}}^{3/4},
\end{equation}
and we deduce that
\begin{equation*}
    \begin{aligned}
        \|V\cdot\nabla V\|_{L^1\dot B_{2,1}^{0}}\leq& C\int_0^T \|V\|_{\dot B_{2,1}^{1/2}}\|\nabla V\|_{\dot B_{2,1}^{1/2}}\de t
            \leq C\int_0^T \|V\|_{\dot B_{2,1}^{1/2}}\|\nabla V\|_{\dot B_{2,1}^{-1}}^{1/4}\|\nabla^2 V\|_{\dot B_{2,1}^{0}}^{3/4}\de t\\
            \leq & \frac{C}{\epsilon^3\mu^3}\int_0^T \|V\|_{\dot B_{2,1}^{1/2}}^{4}\|V\|_{\dot B_{2,1}^{0}}\de t+\epsilon\mu \|\nabla^2 V\|_{L^1\dot B_{2,1}^{0}}.
    \end{aligned}
\end{equation*}
By a similar argument, we have that
\begin{equation*}
    \begin{aligned}
        & \|B\cdot\nabla B\|_{L^1\dot B_{2,1}^{0}}\leq \frac{C}{\epsilon^3\eta^3}\int_0^T \|B\|_{\dot B_{2,1}^{1/2}}^{4}\|B\|_{\dot B_{2,1}^{0}}\de t+\epsilon\eta \|\nabla^2 B\|_{L^1\dot B_{2,1}^{0}}, \\
        &\|V\cdot\nabla B\|_{L^1\dot B_{2,1}^{0}}\leq \frac{C}{\epsilon^3\eta^3}\int_0^T \|V\|_{\dot B_{2,1}^{1/2}}^{4}\|B\|_{\dot B_{2,1}^{0}}\de t+\epsilon\eta \|\nabla^2 B\|_{L^1\dot B_{2,1}^{0}}, \\
         &\|B\cdot\nabla V\|_{L^1\dot B_{2,1}^{0}}\leq \frac{C}{\epsilon^3\mu^3}\int_0^T \|B\|_{\dot B_{2,1}^{1/2}}^{4}\|V\|_{\dot B_{2,1}^{0}}\de t+\epsilon\mu \|\nabla^2 V\|_{L^1\dot B_{2,1}^{0}}.
    \end{aligned}
\end{equation*}
Therefore, combining \eqref{maxregV} and \eqref{maxregB} and taking $\epsilon$ small enough, we obtain
\begin{equation*}
    \begin{aligned}
        \|V\|_{L^\infty\dot B^0_{2,1}}&+\|V_t,\mu\nabla^2 V\|_{L^1\dot B^0_{2,1}}+\|B\|_{L^\infty\dot B^0_{2,1}}+\|B_t,\eta\nabla^2 B\|_{L^1\dot B^0_{2,1}}\\
       & \leq C \biggl(\frac{1}{\mu^3}\int_0^T \|V\|_{\dot B_{2,1}^{1/2}}^4\|V\|_{\dot B_{2,1}^{0}}\de t 
        +\frac{1}{\eta^3}\int_0^T \|B\|_{\dot B_{2,1}^{1/2}}^4\|B\|_{\dot B_{2,1}^{0}}\de t\\
       &+\frac{1}{\mu^3}\int_0^T \|B\|_{\dot B_{2,1}^{1/2}}^4\|V\|_{\dot B_{2,1}^{0}}\de t
       +\frac{1}{\eta^3}\int_0^T \|V\|_{\dot B_{2,1}^{1/2}}^4\|B\|_{\dot B_{2,1}^{0}}\de t+\|V_0\|_{\dot B_{2,1}^{0}}+\|B_0\|_{\dot B_{2,1}^{0}} \biggr).
    \end{aligned}
\end{equation*}
Thus, by Gronwall Lemma we get
\begin{equation}\label{bound1M}
    \begin{aligned}
         \|V\|_{L^\infty\dot B^0_{2,1}}&+\|V_t,\mu\nabla^2 V\|_{L^1\dot B^0_{2,1}}+\|B\|_{L^\infty\dot B^0_{2,1}}+\|B_t,\eta\nabla^2 B\|_{L^1\dot B^0_{2,1}}\\
         &\leq C\left( \|V_0\|_{\dot B_{2,1}^0}+\|B_0\|_{\dot B_{2,1}^0} \right)\text{exp}\biggl( C\biggl( \frac{1}{\mu^3}+\frac{1}{\eta^3}\biggr)\int_0^T \|V\|_{\dot B_{2,1}^{1/2}}^4+\|B\|_{\dot B_{2,1}^{1/2}}^4 \de t\biggr).
    \end{aligned}
\end{equation}
From the interpolation inequality \eqref{interpL4B1/2} and the energy balance 
\begin{align}
    \|V(t)\|_{L^2}^2+\|B(t)\|_{L^2}^2+2\mu\int_0^t\|\nabla V(\tau)\|_{L^2}^2\de \tau+2\eta\int_0^t\|\nabla B(\tau)\|_{L^2}^2\de \tau=\|V_0\|_{L^2}^2+\|B_0\|_{L^2}^2
\end{align}
it follows that
\begin{equation*}
    \mu^{1/4}\|V\|_{L^4\dot B_{2,1}^{1/2}}+\eta^{1/2}\|B\|_{L^4\dot B_{2,1}^{1/2}}\leq C\left(\|V_0\|_{L^2}+\|B_0\|_{L^2} \right).
\end{equation*}
Combining this latter with \eqref{bound1M}, we finally get \eqref{stima espli M dim2}. This concludes the proof.
\end{proof}
We conclude this section by recalling some basic results concerning the well-posedness of the system \eqref{MHDcomp}. First, we recall the local well-posedness result in \cite[Theorem 1.1]{JPG}.
\begin{theorem}\label{thm:lwp}
    Let $\bar\rho>0$ and $c_0>0.$ Assume that the initial data $(\rho_0,v_0,b_0)$ satisfy
    \begin{equation*}
        \rho_0-\bar\rho\in \dot B_{p,1}^{d/p}, \quad c_0\leq\rho_0\leq c_0^{-1},\quad v_0,\ b_0\in \dot B_{p,1}^{d/p-1}, \quad \dive b_0=0.
    \end{equation*}
    Then, there exists a positive time $T>0$ such that if $p\in [2,2d)$, the system \eqref{MHDcomp} has a unique solution $(\rho-\bar\rho,v,b),$ which satisfies
    \begin{equation*}
        \begin{aligned}
            &\rho-\bar\rho\in \tilde C([0,T];\dot B_{p,1}^{d/p}), \quad \frac{1}{2}c_0\leq \rho\leq 2c_0^{-1},\\
            &v,\ b\in \tilde C([0,T];\dot B_{p,1}^{d/p-1})\cap L^1([0,T];\dot B_{p,1}^{d/p+1}),
        \end{aligned}
    \end{equation*}
where $\tilde C([0,T];\dot B_{p,1}^{d/p-1})=\tilde L_T^\infty(\dot B_{p,1}^{d/p-1})\cap C([0,T];\dot B_{p,1}^{d/p-1})$.
\end{theorem}
Lastly, we recall the continuation criterion proved in \cite[Proposition 2.6]{LMW}.
\begin{theorem}\label{thm:continuation}
    Assume that the system \eqref{MHDcomp} has a solution $(\rho,v,b)$ on $[0,T)\times \R^d$ which, for all $ T'<T$, belongs to 
    \begin{equation*}
        E_{T'}:=\tilde C([0,T];\dot B_{2,1}^{d/2})\times \tilde C([0,T];\dot B_{2,1}^{d/2-1})\cap L^1([0,T];\dot B_{2,1}^{d/2+1})\times\tilde C([0,T];\dot B_{2,1}^{d/2-1})\cap L^1([0,T];\dot B_{2,1}^{d/2+1}),
    \end{equation*}
    and satisfies
    \begin{equation*}
    \begin{aligned}
          &\rho-1\in L^{\infty}((0,T);\dot B_{2,1}^{d/2}), \quad  \inf_{(t,x)\in[0,T)\times\R^d} \rho(t,x)>0,\\
        &\int_0^T \|\nabla v\|_{\dot B_{2,1}^{d/2}}\de t<\infty, \quad \int_0^T \|\nabla b\|_{\dot B_{2,1}^{d/2}}\de t<\infty.
    \end{aligned}
    \end{equation*}
    Then there exists a $T^*>T$ such that $(\rho,v,b)$ may be extended to a solution of \eqref{MHDcomp} on $[0,T^*]\times \R^d$ which belongs to $E_{T^*}.$
\end{theorem}

\section{Proof of Theorem \ref{thm:main}}
\label{proof}
In this section, we provide the proof of our first main result, namely Theorem \ref{thm:main}. We recall the statement below for reader's convenience.
\MainTheoremA*
\begin{proof}
We divide the proof in several steps.\\
\\
{\em \underline{Step 0}\quad Local well-posedness.}\\
\\
First of all, in view of the local well-posedness result in Theorem \ref{thm:lwp}, 
we can infer that there exist some $T_*>0$ and a unique maximal solution $(\rho,v,b)$ to \eqref{MHDcomp} satisfying
\begin{equation}\label{regcond}
    a:=\rho-1 \in \ C([0,T_*);\dot B_{2,1}^{d/2}(\R^d)), \quad v,b \in \ C([0,T_*);\dot B_{2,1}^{d/2-1}(\R^d))\cap L^1([0,T_*);\dot B_{2,1}^{d/2+1}(\R^d)).
\end{equation}
Thus, our goal is to derive global-in-time a priori estimates that will allow us to extend this local solution to a global one relying on the continuation criterion in Theorem \ref{thm:continuation}
To do this, we compare the solutions of \eqref{MHDcomp} and \eqref{MHDincomp}. 
Recall that \eqref{MHDincomp} admits a unique global solution when $d=2$, 
whereas in dimension $d=3$ we simply assume that the corresponding incompressible solution is global.  
With this understood, the proof proceeds in the same way for $d=2$ and $d=3$: although the constants in the estimates may depend on the dimension, the structure of the argument does not.  
The only genuine difference between the two cases is thus the global well-posedness of the limiting incompressible system. In the following, we will consider an arbitrary $T\in (0,T_*)$ and when denoting space-time norms, we will explicitly omit the time interval $(0,T)$. It is understood that all norms are computed over this interval, since the compressible solution is a priori defined only up to time $T_*$. This convention is adopted to simplify the notation.\\
\\
{\em\underline{Step 1}\quad The system for the differences.}\\
\\
Given the reference solution $(V,B,\Pi)$, we define the differences
\begin{equation}
u := v-V, \qquad h:= b-B, \qquad a:= \rho-1.
\end{equation}
Then, we rewrite the system \eqref{MHDcomp} in the unknowns $(a,u,h)$ as
\begin{equation}\label{eq:differenze}
    \begin{cases}
        a_t+\dive(au)+\dive \cq u+V\cdot\nabla a=0,\\
        u_t+u\cdot\nabla u +V\cdot\nabla u+u\cdot\nabla V-\mu\Delta u-(\lambda+\mu)\nabla \dive u+\nabla p\\
        \qquad\qquad\qquad\qquad=h\cdot\nabla h+B\cdot\nabla h+h\cdot\nabla B-\frac12\nabla|h+B|^2+\mathcal{R},\\
        h_t+ (h+B)\dive u+(u+V)\cdot\nabla h+u\cdot\nabla B =h\cdot\nabla(u+V)+B\cdot\nabla u+\eta\Delta h,\\
        \dive h=0,\\
        a|_{t=0}=a_0,\quad u|_{t=0}=v_0-V_0,\quad h|_{t=0}=0.
    \end{cases}
\end{equation}
where the remainder is given by
\begin{equation}
    \label{remainder globale}
    \mathcal{R}:=a[u_t+V_t+(u+V)\cdot\nabla u+u\cdot\nabla V],
\end{equation}
while the pressure term is defined as
\begin{equation}
    \label{pressione differenze}
    p=P-\Pi-\frac12|B|^2,
\end{equation}
and we assume, for simplicity, that $P'(1)=1$ and we use the notation
\begin{equation}
    \label{def:k}
    k(a):=P'(1+a)-P'(1)=P'(1+a)-1.
\end{equation}
For simplicity, from now on we assume that $\mu=1,\eta=1$.
We need to consider separately the divergence-free $\mathcal P u$ and the potential $\mathcal Q u$ part of the velocity field $u$. To do this, we apply the Leray projector $\mathcal{P}$ to the equation $\eqref{eq:differenze}_2$ and we get that
\begin{align}
    (\cp u)_t+\cp((u+V)\cdot\nabla\cp u)-\Delta\cp u=&-\cp(a(V_t+\cp u_t+(\cq u_t+\nabla a))-\cp R_1 \notag \\
    &+\cp(h\cdot\nabla B)+\cp(B\cdot\nabla h)+\cp(h\cdot\nabla h)
\end{align}
where the remainder term $R_1$ is defined as
\begin{align}
    R_1:&= (1+a)\cp u\cdot\nabla(V+\cq u)+(1+a)V\cdot\nabla\cq u+(1+a)\cq u\cdot\nabla V \nonumber\\
    &+a(u+V)\cdot\nabla\cp u+aV\cdot\nabla V+a\cq u\cdot\nabla\cq u.\label{def:R1}
\end{align}
Notice that to get the above expression we used the fact that $\mathcal{P}(a\nabla a)=0$. On the other hand, applying the operator $\mathcal{Q}$ to the equation $\eqref{eq:differenze}_2$ we obtain
\begin{align}
    \cq u_t+\cq((u+V)\cdot\nabla\cq u)-\nu\Delta\cq u+\nabla a
   =& -\cq(au_t+aV_t)+\cq((h+B)\cdot \nabla (h+B)) \nonumber\\
   &-\cq R_2-\frac{1}{2}\nabla(|h+B|^2),\label{3.17NEW}
\end{align}
where we used that $\nabla\dive v=\Delta(-(-\Delta)^{-1})\nabla\dive v=\Delta\mathcal{Q}v$, and the remainder term $R_2$ is defined as
\begin{equation}\label{def:R2}
R_2:= (1+a)(u+V)\cdot\nabla\cp u+(1+a)(u+V)\cdot\nabla V+a(u+V)\cdot \nabla\cq u+k(a)\nabla a
\end{equation}

From now on, we denote by $a^\ell$ and $a^h$ the low and high frequencies parts of $a$ as in \eqref{def:split frequenze}. Moreover, we define the following quantities
\begin{equation}\label{def:quantità}
\begin{aligned}
    & \mathcal{X}(\tau):=\|\cq u, a, \nu\nabla a\|_{L^{\infty}((0,\tau;\dot B^{d/2-1}_{2,1})}, \\
    & \mathcal{Y}(\tau):=\|\cq u_t+\nabla a, \nu \nabla^2\cq u, \nu\nabla^2a^\ell,\nabla a^h\|_{L^{1}((0,\tau);\dot B^{d/2-1}_{2,1})},\\
    & \mathcal{Z}(\tau):=\|\cp u\|_{L^{\infty}((0,\tau);\dot B^{d/2-1}_{2,1})}+\|h\|_{L^{\infty}((0,\tau);\dot B^{d/2-1}_{2,1})},\\
    &\mathcal{W}(\tau):=\|\cp u_t,\nabla^2 \cp u\|_{L^{1}((0,\tau);\dot B^{d/2-1}_{2,1})}+\|h_t,\nabla^2 h\|_{L^{1}((0,\tau);\dot B^{d/2-1}_{2,1})}.
\end{aligned}
\end{equation}
We omit the time dependence when the above quantities are computed over $(0,T)$, that is $\mathcal{X}(T)=\mathcal{X}$.
Now, we fix some $M\geq0$ so that the incompressible solution $(V,B)$ to $\eqref{MHDincomp}$ satisfies, for all $T\geq0$
\begin{equation}\label{bound V}
        \mathcal{V}:=\|V\|_{L^{\infty}\dot B^{d/2-1}_{2,1}}+\|V_t,\nabla^2V\|_{L^{1}\dot B^{d/2-1}_{2,1}}+\|B\|_{L^{1}\dot B^{d/2+1}_{2,1}}+\|B\|_{L^{\infty}\dot B^{d/2-1}_{2,1}}\leq M.
\end{equation}
We claim that for $\nu$ large enough, we can find some large $D$ and small $\delta$ so that for all $T<T_{*}$, the following bounds hold
\begin{equation}\label{local bounds X Y Z W}
    \mathcal{X}+\mathcal{Y}\leq D \quad \text{and}\quad \mathcal{Z}+\mathcal{W}\leq\delta.
\end{equation}
\\
{\em \underline{Step 2} \quad Estimates for the divergence-free part of the system \eqref{eq:differenze}}.\\
\\
In this step we look for estimates on $\mathcal{Z}+\mathcal{W}(T)$, hence on quantities which involve only divergence-free part of the system \eqref{eq:differenze}, which we rewrite in the following system
\begin{equation}\label{Peq}
\begin{cases}
(\cp u)_t+\cp((u+V)\cdot\nabla\cp u)-\Delta\cp u=-\cp(a(V_t+\cp u_t+(\cq u_t+\nabla a))-\cp R_1+\cp(h\cdot\nabla B)\\
\hspace{6.5cm}+\cp(B\cdot\nabla h)+\cp(h\cdot\nabla h),\\
h_t+\cp(\dive(\cq u)h)+\cp(\dive u B)+\cp(u\cdot\nabla B)+\cp(\cp u\cdot\nabla h+\cq u\cdot\nabla h)+\cp(V\cdot\nabla h)-\cp(B\cdot\nabla u)\\
\hspace{6.5cm}-\cp(h\cdot\nabla\cp u)-\cp(h\cdot\nabla\cq u)-\cp(h\cdot\nabla V)=\cp \Delta h,\\
\cp u|_{t=0}=0,\qquad h|_{t=0}=0.
\end{cases}
\end{equation}
We apply $\dot\Delta_j$ to $\eqref{Peq}_1$, we multiply by $\dot\Delta_j\cp u$ and we integrate over $\R^d$ to obtain that
\begin{align*}
\frac{1}{2}\frac{\de}{\de t} \|\dot\Delta_j\cp u\|_{L^2}^2&+\|\nabla\dot\Delta_j\cp u\|_{L^2}^2+\int_{\R^d} \dot\Delta_j((u+V)\cdot\nabla\cp u)\cdot\dot\Delta_j\cp u \de x\\
=&-\int_{\R^d} \dot\Delta_j(a(V_t+\cp u_t+(\cq u_t+\nabla a))+R_1)\cdot\dot\Delta_j\cp u \de x\\
&+\int_{\R^d} \dot\Delta_j(h\cdot\nabla B) \cdot\dot\Delta_j\cp u \de x+\int_{\R^d} \dot\Delta_j (B\cdot\nabla h)\cdot\dot\Delta_j\cp u \de x \\
&+\int_{\R^d} \dot\Delta_j(h\cdot\nabla h)\cdot\dot\Delta_j\cp u \de x,
\end{align*}
where we used that the operators $\cp$ and $\dot\Delta_j$ commute, that $\cp$ is self-adjoint in $L^2$ and that $\cp^2=\cp$.
We use the Bernstein inequality in \ref{Bernstein} to estimate the second term above as
\begin{equation}\label{stimda alto grad}
      \|\nabla\dot\Delta_j\cp u\|_{L^2}^2=\|\nabla\dot\Delta_j\cp u\|_{L^2}\|\nabla\dot\Delta_j\cp u\|_{L^2}\geq c\|\nabla^2\dot\Delta_j\cp u\|_{L^2}\|\dot\Delta_j\cp u\|_{L^2}.
\end{equation}
Then, we use the commutator $[u + V,\dot\Delta_j]$ to get (after integration by parts) that
\begin{align*}
\frac{1}{2}\frac{\de}{\de t} \|\dot\Delta_j\cp u\|_{L^2}^2+c\|\nabla^2\dot\Delta_j\cp u\|_{L^2}\|\dot\Delta_j\cp u\|_{L^2}\leq \frac{1}{2}\int_{\R^d}|\dot\Delta_j\cp u|^2 \dive u\de x +\int_{\R^d}([u+V,\dot\Delta_j]\cdot\nabla\cp u)\cdot\dot\Delta_j\cp u \de x\\
-\int_{\R^d} \dot\Delta_j(a(V_t+\cp u_t+(\cq u_t+\nabla a))+R_1)\cdot\dot\Delta_j\cp u \de x +\int_{\R^d} \dot\Delta_j(h\cdot\nabla B) \cdot\dot\Delta_j\cp u \de x\\
+\int_{\R^d} \dot\Delta_j (B\cdot\nabla h)\cdot\dot\Delta_j\cp u \de x +\int_{\R^d} \dot\Delta_j(h\cdot\nabla h)\cdot\dot\Delta_j\cp u \de x.
\end{align*}
Now, by using the embedding $\dot B_{2,1}^{d/2}(\R^d)\hookrightarrow L^\infty(\R^d)$, we have
\begin{equation*}
\int_{\R^d} |\dot\Delta_j\cp u|^2\dive u \de x\leq C \|\dot\Delta_j\cp u\|_{L^2}^2\|\nabla u\|_{\dot B_{2,1}^{d/2}}.
\end{equation*}
So, we use Holder inequality and we divide by $\|\dot\Delta_j\cp u\|_{L^2(\R^d)}$ to get the inequality
\begin{align}
\frac{\de}{\de t} \|\dot\Delta_j\cp u\|_{L^2}&+c\|\nabla^2\dot\Delta_j\cp u\|_{L^2}\lesssim \|\dot\Delta_j\cp u\|_{L^2}\|\nabla u\|_{\dot B_{2,1}^{d/2}}+\|[u+V,\dot\Delta_j]\cdot\nabla\cp u\|_{L^2}\nonumber\\
&+\|\dot\Delta_j(aV_t+\cp u_t+(\cq u_t+\nabla a))\|_{L^2} +\|R_1\|_{L^2}+\|\dot\Delta_j(h\cdot\nabla B)\|_{L^2}\nonumber\\
&+\|\dot\Delta_j(B\cdot\nabla h)\|_{L^2}+\|\dot\Delta_j(h\cdot\nabla h)\|_{L^2}.\label{est:incomp-dopo holder}
\end{align}
We use Lemma \ref{lemma commut 2.100} to bound the commutator term as
\begin{equation}
\bigl\|\bigl[u+V,\dot\Delta_j]\cdot\nabla\cp u\bigr\|_{L^2}\leq C c_j2^{-j({d/2-1})}\|\nabla(u+V)\|_{\dot B^{d/2}_{2,1}}\|\cp u\|_{\dot B^{d/2-1}_{2,1}}
\quad\text{with}\quad \sum_{j\in\Z} c_j\leq1,
\end{equation}
and then, we multiply the inequality \eqref{est:incomp-dopo holder} by $2^{j(d/2-1)} $, we integrate in time, and we sum over $j\in\mathbb{Z}$, to obtain for any $T<T_*$
\begin{equation}\label{stimaPu,grad2Pu}
\begin{aligned}
\sup_{t\in[0,T]}\|\cp u\|_{\dot B_{2,1}^{d/2-1}}&+\int_0^T\|\nabla^2\cp u\|_{\dot  B_{2,1}^{d/2-1}}\de t\lesssim
\int_0^T \|\nabla u\|_{\dot B_{2,1}^{d/2}}\|\cp u\|_{\dot B_{2,1}^{d/2-1}} \de t\\
&+ \int_0^T \|\nabla(u+V)\|_{\dot B_{2,1}^{d/2}}\|\cp u\|_{\dot B_{2,1}^{d/2-1}} \de t+\int_0^T \|a(V_t+\cp u_t+(\cq u_t+\nabla a))\|_{\dot B_{2,1}^{d/2-1}} \de t\\
&+\int_0^T \|R_1\|_{\dot B_{2,1}^{d/2-1}} \de t+\int_0^T \|h\cdot\nabla B\|_{\dot B_{2,1}^{d/2-1}} \de t+\int_0^T \|B\cdot\nabla h\|_{\dot B_{2,1}^{d/2-1}} \de t\\
&+\int_0^T \|h\cdot\nabla h\|_{\dot B_{2,1}^{d/2-1}} \de t, 
\end{aligned}
\end{equation}   
where in particular we used that $\cp u|_{t=0}=0$. Since our goal is to bound the quantity $\mathcal{W}$, we need to find an estimate for $\cp u_t$. To do this, we directly use the equation $\eqref{Peq}_1$ to get
\begin{align*}
    \int_0^T\|\cp u_t\|_{B_{2,1}^{d/2-1}}\de t\leq &  \int_0^T \|(u+V)\cdot\nabla\cp u\|_{\dot B_{2,1}^{d/2-1}}\de t+\int_0^T\|\nabla^2\cp u\|_{\dot B_{2,1}^{d/2-1}}\de t\\
    &+\int_0^T\|a(V_t+\cp u_t
    +(\cq u_t+\nabla a))\|_{\dot B_{2,1}^{d/2-1}}\de t
    +\int_0^T\| R_1\|_{\dot B_{2,1}^{d/2-1}}\de t\\
    &+\int_0^T\|h\cdot\nabla B\|_{\dot B_{2,1}^{d/2-1}}\de t
    +\int_0^T\|B\cdot\nabla h\|_{\dot B_{2,1}^{d/2-1}}\de t
    +\int_0^T\|h\cdot\nabla h\|_{\dot B_{2,1}^{d/2-1}}\de t,
\end{align*}
where we used that the operator $\cp:\dot B_{2,1}^{d/2-1}\to\dot B_{2,1}^{d/2-1}$ is bounded.
Hence, combining this with the estimate \eqref{stimaPu,grad2Pu}, we have
\begin{equation}\label{3.27New}
\begin{aligned}
\|\cp u\|_{L^\infty\dot B_{2,1}^{d/2-1}}&+\|\cp u_t,\nabla^2\cp u\|_{L^1\dot  B_{2,1}^{d/2-1} }\lesssim\int_0^T \|(u+V)\cdot\nabla\cp u\|_{\dot B_{2,1}^{d/2-1}}\de t\\
&+\int_0^T \|\nabla u\|_{\dot B_{2,1}^{d/2}}\|\cp u\|_{\dot B_{2,1}^{d/2-1}} \de t+\int_0^T \|\nabla(u+V)\|_{\dot B_{2,1}^{d/2}}\|\cp u\|_{\dot B_{2,1}^{d/2-1}} \de t\\
&+\int_0^T (\|a(V_t+\cp u_t+(\cq u_t+\nabla a))\|_{\dot B_{2,1}^{d/2-1}} +\|R_1\|_{\dot B_{2,1}^{d/2-1}} )\de t\\
&+\int_0^T (\|h\cdot\nabla B\|_{\dot B_{2,1}^{d/2-1}} +\|B\cdot\nabla h\|_{\dot B_{2,1}^{d/2-1}} + \|h\cdot\nabla h\|_{\dot B_{2,1}^{d/2-1}}) \de t. 
\end{aligned}
\end{equation}
We now estimate the terms on the right-hand side of \eqref{3.27New}, we denote them by $(I),...,(V)$. For the terms $(I)-(IV)$ we proceed as in \cite{danchin2016compressiblenavierstokeslarge}. For the reader's convenience, we recall below the corresponding estimates.
From now on, we assume that
\begin{equation}\label{3.29}
    \mathcal{X}\ll \nu.
\end{equation}
To bound $(I)$, we use Lemma \eqref{lemmastimaprod}, the interpolation inequality \eqref{interpolation} and Young's inequality, and we get
\begin{equation*}
\begin{aligned}
     (I)=\int_0^T\|(u+V)\cdot\nabla\cp u\|_{\dot B_{2,1}^{d/2-1}}\de t\lesssim&
    \int_0^T\left(\|\nabla\cp u\|_{\dot B_{2,1}^{d/2}}+\|\cq u,V\|^2_{\dot B_{2,1}^{d/2}} \right)\mathcal{Z}(t)\de t+\varepsilon \mathcal{W}.    
\end{aligned}
\end{equation*}
We split $u$ into its compressible and incompressible parts, and using again Lemma \eqref{lemmastimaprod} we have
\begin{equation*}
    \begin{aligned}
    (II)+
    (III)&\leq \int_0^T (\|\nabla u\|_{\dot B_{2,1}^{d/2}}+\|\nabla V\|_{\dot B_{2,1}^{d/2}})\|\cp u\|_{\dot B_{2,1}^{d/2-1}} \de t\lesssim
    \int_0^T\|(\nabla V,\nabla\cp u,\nabla\cq u)\|_{\dot B_{2,1}^{d/2}}\mathcal{Z}(t)\de t,\\
    (IV)&=\int_0^T\|a(V_t+\cp u_t+(\cq u_t+\nabla a))\|_{\dot B_{2,1}^{d/2-1}}\de t
    \lesssim \nu^{-1}(\mathcal{Y}+\mathcal{W}+\mathcal{V}\mathcal{X}).
    \end{aligned}
\end{equation*}
Finally, for the remainder term we split it as
\begin{equation*}
\begin{aligned}
R_1&=(1+a)\cp u\cdot\nabla(V+\cq u)+(1+a)(V\cdot\nabla \cq u+\cq u\cdot\nabla V )+a(u+V)\cdot\nabla\cp u\\
&+(a(V\cdot\nabla V+\cq u\cdot\nabla \cq u )=R_1^{(1)}+R_1^{(2)}+R_1^{(3)}+R_1^{(4)},
\end{aligned}
\end{equation*}
and estimating separately the four terms above, one obtains
\begin{equation*}
    \int_0^T\|R_1\|_{\dot B^{\frac{d}{2}-1}_{2,1}}\de t\lesssim\int_0^T \|(\nabla V,\nabla \cq u)\|_{\dot B_{2,1}^{d/2} }\mathcal{Z}(t)\de t+\nu^{-1/2}\mathcal{X}^{1/2}\mathcal{Y}^{1/2}\mathcal{V}(T)+\nu^{-1}\mathcal{X}(\mathcal{V}^2+\nu^{-1}\mathcal{X}\mathcal{Y}).
\end{equation*}
Lastly, to bound the term $(V)$, we apply Lemma \ref{lemmastimaprod} and we get
\begin{equation*}
    (V)\lesssim \int_0^T (\|\nabla h\|_{\dot B_{2,1}^{d/2}}+\|B\|_{\dot B_{2,1}^{d/2+1} })\mathcal{Z}(t)\de t
\end{equation*}
Therefore, we have obtained
\begin{equation}\label{stima Pu}
    \begin{aligned}
        \|\cp u\|_{L^\infty\dot B_{2,1}^{d/2-1}}&+\|\cp u_t,\nabla^2\cp u\|_{L^1\dot  B_{2,1}^{d/2-1} }\\
        \lesssim&\int_0^T \biggl( \|\nabla V,\nabla\cp u, \nabla\cq u,\nabla h\|_{\dot B_{2,1}^{d/2}}+\|\cq u,V\|_{\dot B_{2,1}^{d/2}}^2
        +\|B\|_{\dot B_{2,1}^{d/2+1} }\biggr) \mathcal{Z}(t)\de t\\
        &+\nu^{-1} (\mathcal{Y}+\mathcal{W}+\mathcal{V}\mathcal{X}+\nu^{-1/2}\mathcal{X}^{1/2}\mathcal{Y}^{1/2}\mathcal{V}(1+\nu^{-1}\mathcal{X})\\
        &+\nu^{-1}\mathcal{X}\bigl(\mathcal{Z}+\mathcal{X}+\mathcal{V}\bigr)\mathcal{W}
        +\nu^{-1}\mathcal{X}(\mathcal{V}^2+\nu^{-1}\mathcal{X}\mathcal{Y}).
    \end{aligned}
\end{equation}
We now derive analogous estimates for the magnetic field $h$, in the norms appearing in $\mathcal{Z}$ and $\mathcal{W}$. We apply the operator $\dot\Delta_j$ to the equation $\eqref{Peq}_2$, we multiply for $\dot\Delta_j h$ and we integrate in space to get that
\begin{align*}
\frac{1}{2}\frac{\de}{\de t}\|\dot\Delta_j h\|_{L^2(\R^d)}^2
&+\int_{\R^d} \dot\Delta_j (\dive(\cq u)h) \cdot \dot\Delta_j h\de x =-\int_{\R^d} \dot\Delta_j (\dive u B)\cdot \dot\Delta_j h\de x\\
& -\int_{\R^d} \dot\Delta_j (u\cdot\nabla B)\cdot \dot\Delta_j h\de x - \int_{\R^d} \dot\Delta_j (\cp u\cdot\nabla h+\cq u\cdot\nabla h) \cdot \dot\Delta_j h\de x\\
& -\int_{\R^d} \dot\Delta_j(V\cdot\nabla h)\cdot \dot\Delta_j h\de x + \int_{\R^d} \dot\Delta_j(B\cdot\nabla u)\cdot \dot\Delta_j h\de x + \int_{\R^d} \dot\Delta_j(h\cdot\nabla \cp u)\cdot \dot\Delta_j h\de x\\
&+\int_{\R^d} \dot\Delta_j(h\cdot\nabla \cq u)\cdot \dot\Delta_j h\de x +\int_{\R^d} \dot\Delta_j(h\cdot\nabla V)\cdot \dot\Delta_j h\de x +\int_{\R^d}  \dive(\nabla\dot\Delta_j h)\cdot \dot\Delta_j h\de x.
\end{align*}
We integrate by parts the term involving $\dive(\nabla \dot\Delta_j h)$, we use Holder and Bernstein inequalities, and we divide by $\|\dot\Delta_j h\|_{L^2}$ (proceeding similar to \eqref{stimda alto grad}), to get
\begin{align*}
\frac{\de}{\de t}\|\dot\Delta_j h\|_{L^2}+c\|\nabla^2\dot\Delta_j h\|_{L^2}&\leq \|\dot\Delta_j (\dive(\cq u)h)\|_{L^2} +\|\dot\Delta_j (\dive u B)\|_{L^2}+\|\dot\Delta_j (u\cdot\nabla B)\|_{L^2}\\
& +\|\dot\Delta_j (\cp u\cdot\nabla h+\cq u\cdot\nabla h)\|_{L^2}
+\|\dot\Delta_j(V\cdot\nabla h)\|_{L^2}+\|\dot\Delta_j(B\cdot\nabla u)\|_{L^2}\\
&+\| \dot\Delta_j(h\cdot\nabla \cp u)\|_{L^2}+\|\dot\Delta_j(h\cdot\nabla \cq u)\|_{L^2}+\|\dot\Delta_j(h\cdot\nabla V)\|_{L^2}.
\end{align*}
We multiply by $2^{j(d/2-1)}$, we sum over $j$ and we integrate in time so that for all $T<T_*$
\begin{equation}\label{stima1h}
\begin{aligned}
\|h\|_{L^\infty\dot B_{2,1}^{d/2-1} }&+\int_0^T\|\nabla^2 h\|_{\dot B_{2,1}^{d/2-1} }\leq \int_0^T \|\dive(\cq u) h\|_{\dot B_{2,1}^{d/2-1}} \de t+\int_0^T\|\dive u B\|_{\dot B_{2,1}^{d/2-1}} \de t\\
&+\int_0^T\| u\cdot\nabla B\|_{\dot B_{2,1}^{d/2-1}} \de t+\int_0^T\|\cp u\cdot\nabla h\|_{\dot B_{2,1}^{d/2-1}} \de t+\int_0^T\| \cq u\cdot\nabla h\|_{\dot B_{2,1}^{d/2-1}} \de t\\
&+\int_0^T\| V\cdot\nabla h\|_{\dot B_{2,1}^{d/2-1}} \de t+\int_0^T\| B\cdot\nabla u\|_{\dot B_{2,1}^{d/2-1}} \de t+\int_0^T\| h\cdot\nabla \cp u\|_{\dot B_{2,1}^{d/2-1}} \de t\\
&+\int_0^T\| h\cdot\nabla \cq u\|_{\dot B_{2,1}^{d/2-1}} \de t+\int_0^T\| h\cdot\nabla V\|_{\dot B_{2,1}^{d/2-1}} \de t.
\end{aligned}
\end{equation}
where we used that $h_0=0$. Lastly, we need an estimate for $\|h_t\|_{L^1\dot B_{2,1}^{d/2-1}}$. To do this, we directly use the equation $\eqref{Peq}_2$, and combine it with $\eqref{stima1h}$ to obtain:
\begin{equation}\label{stima2h}
\begin{aligned}
\|h\|_{L^\infty\dot B_{2,1}^{d/2-1} }&+\|\nabla^2 h,h_t\|_{L^1\dot B_{2,1}^{d/2-1} }\leq \int_0^T \|\dive(\cq u) h\|_{\dot B_{2,1}^{d/2-1}} \de t +\int_0^T\|\dive u B\|_{\dot B_{2,1}^{d/2-1}} \de t\\
&+\int_0^T\| u\cdot\nabla B\|_{\dot B_{2,1}^{d/2-1}} \de t+\int_0^T\|\cp u\cdot\nabla h\|_{\dot B_{2,1}^{d/2-1}} \de t+\int_0^T\| \cq u\cdot\nabla h\|_{\dot B_{2,1}^{d/2-1}} \de t\\
&+\int_0^T\| V\cdot\nabla h\|_{\dot B_{2,1}^{d/2-1}} \de t+\int_0^T\| B\cdot\nabla u\|_{\dot B_{2,1}^{d/2-1}} \de t+\int_0^T\| h\cdot\nabla \cp u\|_{\dot B_{2,1}^{d/2-1}} \de t\\
&+\int_0^T\| h\cdot\nabla \cq u\|_{\dot B_{2,1}^{d/2-1}} \de t+\int_0^T\| h\cdot\nabla V\|_{\dot B_{2,1}^{d/2-1}} \de t.
\end{aligned}
\end{equation}
We now estimate the terms on the right-hand side of \eqref{stima2h}. To avoid further notation, we denote them by $(I),...,(X)$. The terms $(I)-(IV)$ are handled using Lemma \ref{lemmastimaprod} and Lemma \ref{lem:bernstein} as follows
\begin{align*}
(I)=&\int_0^T \|\dive(\cq u) h\|_{\dot B_{2,1}^{d/2-1}} \de t\lesssim\int_0^T \|\nabla\cq u \|_{\dot B_{2,1}^{d/2}} \| h\|_{\dot B_{2,1}^{d/2-1}}\de t\lesssim\int_0^T \|\nabla\cq u \|_{\dot B_{2,1}^{d/2}} \mathcal{Z}(t) \de t,\\
(II)=&\int_0^T\|\dive u B\|_{\dot B_{2,1}^{d/2-1}} \de t \lesssim\int_0^T \|\nabla^2\cq u \|_{\dot B_{2,1}^{d/2-1}} \| B\|_{\dot B_{2,1}^{d/2-1}}\de t\leq \nu^{-1}\mathcal{Y}\mathcal{V}(T),\\
(III)=&\int_0^T\| u\cdot\nabla B\|_{\dot B_{2,1}^{d/2-1}} \de t\lesssim\int_0^T\left(\|\cp u\cdot\nabla B\|_{\dot B_{2,1}^{d/2-1}}+\|\cq u\cdot\nabla B\|_{\dot B_{2,1}^{d/2-1}}\right) \de t\\
&\lesssim \int_0^T \|\nabla B\|_{\dot B_{2,1}^{d/2}} \mathcal{Z}(t) \de t+\nu^{-1}\| B\|_{L^\infty \dot B_{2,1}^{d/2-1}}  \|\nu \nabla^2\cq u\|_{L^1\dot B_{2,1}^{d/2-1}}  \\
&\leq \int_0^T \|B\|_{\dot B_{2,1}^{d/2+1}} \mathcal{Z}(t) \de t+\nu^{-1} \mathcal{V}\mathcal{Y},\\
(IV)=&\int_0^T\| \cp u\cdot\nabla h\|_{\dot B_{2,1}^{d/2-1}} \de t\lesssim
\int_0^T\| \cp u\|_{\dot B_{2,1}^{d/2-1}} \|\nabla h\|_{\dot B_{2,1}^{d/2}} \de t\leq \int_0^T \|\nabla h\|_{\dot B_{2,1}^{d/2}} \mathcal{Z}(t) \de t.
\end{align*}
For the fifth term $(V)$, we use the interpolation inequality \eqref{interpolation} and Young's inequality to get
\begin{align*}
(V)
&\lesssim\int_0^T\| \cq u\|_{\dot B_{2,1}^{d/2}}\| h\|_{\dot B_{2,1}^{d/2}} \de t\lesssim\int_0^T\| \cq u\|_{\dot B_{2,1}^{d/2}}\|h\|_{\dot B_{2,1}^{d/2-1}}^{1/2}\|\nabla^2 h\|_{\dot B_{2,1}^{d/2-1}}^{1/2} \de t\\
&\lesssim \varepsilon \|\nabla^2 h\|_{L^1\dot B_{2,1}^{d/2-1} }+\frac{1}{\varepsilon}\int_0^T\| \cq u\|_{\dot B_{2,1}^{d/2}}^2 \mathcal{Z}(t)\de t,
\end{align*}
Notice that, we can absorb the first term above into the left-hand side of \eqref{stima2h}. Arguing the same way, we also get that
\begin{align*}
(VI)=&\int_0^T \|V\cdot\nabla h\|_{\dot B_{2,1}^{d/2-1} }\de t \lesssim\varepsilon\|\nabla^2 h\|_{\dot B_{2,1}^{d/2-1}}+\frac{1}{\varepsilon}\int_0^T \|V\|_{\dot B_{2,1}^{d/2} }^2\|h\|_{\dot B_{2,1}^{d/2-1}}\de t.
\end{align*}
For the seventh term $(VII)$, we separate $u$ into its compressible and incompressible parts.
\begin{equation*}
\int_0^T \|B\cdot\nabla u\|_{\dot B_{2,1}^{d/2-1}}\de t\leq\int_0^T \|B\cdot\nabla \cp u\|_{\dot B_{2,1}^{d/2-1}}\de t+\int_0^T \|B\cdot\nabla \cq u\|_{\dot B_{2,1}^{d/2-1}}\de t,
\end{equation*}
and we estimate the two terms separately as follows
\begin{align*}
\int_0^T \|B\cdot\nabla \cp u\|_{\dot B_{2,1}^{d/2-1}}\de t\lesssim& 
\int_0^T \|B\|_{\dot B_{2,1}^{d/2+1}}\mathcal{Z}(t)\de t,\\
\int_0^T \|B\cdot\nabla \cq u\|_{\dot B_{2,1}^{d/2-1}}\de t\lesssim&  
\int_0^T \|B\|_{\dot B_{2,1}^{d/2-1}}\|\nabla^2\cq u\|_{\dot B_{2,1}^{d/2-1}}\de t
\leq \nu^{-1} \mathcal{V}\mathcal{Y}.
\end{align*}
Finally, to estimate the terms $(XIII)-(X)$ in $\eqref{stima2h}$, we use again Lemma \ref{lemmastimaprod} and Lemma \ref{lem:bernstein} to get  
\begin{align*}
(XIII)=&\int_0^T \|h\cdot\nabla\cp u\|_{\dot B_{2,1}^{d/2-1}}\de t\lesssim\int_0^T \|\nabla\cp u\|_{\dot B_{2,1}^{d/2}}\mathcal{Z}(t) \de t, \\
(IX)=&\int_0^T \|h\cdot\nabla\cq u\|_{\dot B_{2,1}^{d/2-1}}\de t\lesssim \int_0^T \|\nabla\cq u\|_{\dot B_{2,1}^{d/2}}\mathcal{Z}(t) \de t,\\
(X)=&\int_0^T \|h\cdot\nabla V\|_{\dot B_{2,1}^{d/2-1}}\de t\lesssim\int_0^T \|\nabla V\|_{\dot B_{2,1}^{d/2}} \mathcal{Z}(t) \de t.
\end{align*}
Therefore, collecting all the estimates above, we finally obtain 
\begin{align}
\|h\|_{L^\infty\dot B_{2,1}^{d/2-1}}+\|h_t,\nabla^2 h\|_{\dot B_{2,1}^{d/2-1}}\lesssim&\int_0^T (\|\nabla V,\nabla\cp u,\nabla\cq u,\nabla h\|_{\dot B_{2,1}^{d/2}}+\|\cq u,V\|^2_{\dot B_{2,1}^{d/2}}+\|B\|_{\dot B_{2,1}^{d/2+1}}) \mathcal{Z}(t) \de t\nonumber\\
&+\nu^{-1} \mathcal{Y}\mathcal{V}.\label{stima h}
\end{align}

Lastly, by the definition of the energies $\mathcal{Z}$ and $\mathcal{W}$ and collecting the estimates in \eqref{stima Pu} and \eqref{stima h} we get the bound
\begin{equation*}
    \begin{aligned}
        \mathcal{Z}+\mathcal{W}\lesssim& 
        \int_0^T (\|\nabla V,\nabla\cp u,\nabla\cq u,\nabla h\|_{\dot B_{2,1}^{d/2}}+\|\cq u,V\|^2_{\dot B_{2,1}^{d/2}}+\|B\|_{\dot B_{2,1}^{d/2+1}}) \mathcal{Z}(t) \de t\\
        &+\nu^{-1}(\mathcal{Y}+\mathcal{W}+\mathcal{V})\mathcal{X}+\nu^{-1/2}\mathcal{X}^{1/2}\mathcal{V}\mathcal{Y}^{1/2}(1+\nu^{-1}\mathcal{X})\\
        &+\nu^{-1}\mathcal{X}(\mathcal{Z}+\mathcal{X}+\mathcal{V})\mathcal{W}+\nu^{-1}\mathcal{X}(\mathcal{V}^2+\nu^{-1}\mathcal{X}\mathcal{Y})+\nu^{-1}\mathcal{V}\mathcal{Y}.
    \end{aligned}
\end{equation*}
Then, we use \eqref{3.29} 
and Gronwall lemma to conclude that
\begin{equation}\label{local est Z+W}
    \begin{aligned}
        \mathcal{Z}+&\mathcal{W}\leq Ce^{C\left(\|\nabla V,\nabla\cp u,\nabla\cq u,\nabla h\|_{L^1\dot B_{2,1}^{d/2}}+\|\cq u,V\|^2_{L^1\dot B_{2,1}^{d/2}}+\|B\|_{L^1\dot B_{2,1}^{d/2+1}}\right) }\\
        &\times(\nu^{-1}(\mathcal{Y}+\mathcal{W}+\mathcal{V})\mathcal{X}+\nu^{-1/2}\mathcal{X}^{1/2}\mathcal{V}\mathcal{Y}^{1/2}+\nu^{-1}\mathcal{X}(\mathcal{Z}+\mathcal{X}+\mathcal{V})\mathcal{W}+\nu^{-1}\mathcal{X}\mathcal{V}^2+\nu^{-1}\mathcal{V}\mathcal{Y}).
    \end{aligned}
\end{equation}
\\
{\em \underline{Step 3} \quad Estimates on the potential part of the velocity and on the density}.\\
\\
In order to estimate the potential component of the velocity, we consider the momentum and continuity equations together and couple their energy estimates so as to take advantage of specific cancellations.
First, we apply the projector $\dot\Delta_j$ to equation $\eqref{eq:differenze}_1$ and we get that
\begin{equation}
a_{j,t}+(u+V)\cdot\nabla a_j+\dive \cq u_j=g_j, \label{3.31}
\end{equation}
where we have defined 
\begin{align}
   & a_j:=\dot\Delta_j a, \quad \cq u_j:=\dot\Delta_j\cq u,\\
   & g_j:=-\dot\Delta_j(a\dive \cq u)-[\dot\Delta_j,(u+V)]\cdot\nabla a.\label{def:gj}
\end{align}
Secondly, we apply the operator $\dot\Delta_j$ to \eqref{3.17NEW} and we obtain
\begin{equation}
\cq u_{j,t}+\cq((u+V)\cdot\nabla\cq u_j)-\nu\Delta\cq u_j+\nabla a_j-\dot\Delta_j(\cq((B+h)\cdot\nabla(B+h))+\frac{1}{2}\dot\Delta_j(\nabla(|B+h|^2)=f_j, \label{3.32New}
\end{equation}
where
\begin{equation}
    f_j:=-\dot\Delta_j\cq(aV_t+au_t)-\dot\Delta_j\cq R_2-\cq[\dot\Delta_j,u+V]\cdot\nabla\cq u.\label{def:fj}
\end{equation}
Now, we multiply the equations \eqref{3.31} and \eqref{3.32New} by $a_j$ and $\cq u_j$, respectively, and we integrate in space to deduce
\begin{equation}\label{3.36}
    \frac{1}{2}\frac{\de}{\de t}\int_{\R^d} a_j^2 \de x+\int_{\R^d} a_j \dive\cq u_j \de x=\frac{1}{2}\int_{\R^d} \dive u \ a_j^2 \de x+\int_{\R^d} g_j a_j\de x,
\end{equation}
and
\begin{equation}\label{3.37New}
\begin{aligned}
\frac{1}{2}\frac{\de}{\de t}\int_{\R^d} |\cq u_j|^2\de x+&\nu\int_{\R^d} |\nabla \cq u_j|^2\de x-\int_{\R^d} a_j\dive \cq u_j\de x-\int_{\R^d} \dot\Delta_j((h+B)\cdot\nabla(h+B))\cdot\cq u_j\de x\\
&+\frac{1}{2}\int_{\R^d} \dot\Delta_j\nabla(|B+h|^2)\cdot\cq u_j\de x=\frac{1}{2}\int_{\R^d} \dive u |\cq u_j|^2\de x+\int_{\R^d} f_j \cdot\cq u_j\de x.
\end{aligned}
\end{equation}
Now, we look for an estimate on $\nabla a_j$: we differentiate \eqref{3.31} and we get the equation
\begin{equation}\label{3.38}
    \nabla a_{j,t}+(u+V)\cdot\nabla\nabla a_j+\nabla\dive\cq u_j=\nabla g_j-\nabla(u+V)\cdot\nabla a_j.
\end{equation}
We multiply \eqref{3.38} by $\nabla a_j$ and we integrate in space to get
\begin{equation}\label{3.39}
\begin{aligned}
\frac{1}{2}\frac{\de}{\de t}\int_{\R^d} |\nabla a_j|^2\de x+\int_{\R^d} ((u+V)\cdot\nabla\nabla a_j)\cdot\nabla a_j \de x+\int_{\R^d}\nabla\dive \cq u_j\cdot\nabla a_j\de x\\
=\int_{\R^d}(\nabla g_j-\nabla(u+V)\cdot\nabla a_j)\cdot\nabla a_j\de x.
\end{aligned}
\end{equation}
Our goal now is to look for an estimate on the mixed term $\int_{\R^d} \cq u_j\cdot\nabla a_j\de x$, which will allow us to eliminate the term $\nabla\dive\cq u_j$ from \eqref{3.39}. To do this, we test the equation \eqref{3.38} with $\cq u_j$ 
and equation \eqref{3.32New} with $\nabla a_j$,
we sum them and we use the relation $\Delta\cq u_j=\nabla\dive\cq u_j$ to get

\begin{align}
\frac{\de}{\de t}\int_{\R^d} \cq u_j&\cdot\nabla a_j\de x-\int_{\R^d} \cq u_j\cdot\nabla a_j\dive u \de x-\nu\int_{\R^d}\nabla\dive\cq u_j\cdot\nabla a_j\de x+\int_{\R^d} |\nabla a_j|^2\de x\nonumber\\
&+\int_{\R^d} \nabla\dive \cq u_j\cdot\cq u_j\de x-\int_{\R^d}\dot\Delta_j(\cq((B+h)\cdot\nabla(B+h))\cdot\nabla a_j\de x\label{3.40New}\\
&+\frac{1}{2}\int_{\R^d}\dot\Delta_j(\nabla(|B+h|^2)\cdot\nabla a_j\de x=\int_{\R^d} (\nabla g_j-\nabla(u+V)\cdot\nabla a_j)\cdot\cq u_j\de x+\int_{\R^d} f_j\cdot\nabla a_j\de x.\nonumber
\end{align}
We point out that here we have used the identity
\begin{equation}
    \int_{\R^d} (u+V)\cdot\nabla(\cq u_j\cdot\nabla a_j)\de x=-\int_{\R^d} \cq u_j\cdot\nabla a_j\dive u\de x,
\end{equation}
which follows from an integration by parts.
Then, if we multiply the equation \eqref{3.39} by $\nu$ and we sum it with \eqref{3.40New}, 
we can remove the highest-order contributions and we obtain
\begin{align}
\frac{1}{2}\frac{\de}{\de t}\int_{\R^d} &(\nu|\nabla a_j|^2+2\cq u_j\cdot\nabla a_j)\de x+\int_{\R^d} (|\nabla a_j|^2-|\nabla\cq u_j|^2)\de x+\frac{1}{2}\int_{\R^d} \dot\Delta_j(\nabla(|B+h|^2)\cdot\nabla a_j\de x\nonumber\\
&-\int_{\R^d} \dot\Delta_j(\cq (B+h)\cdot\nabla(B+h))\cdot\nabla a_j\de x=\int_{\R^d} \left(\frac{\nu}{2}|\nabla a_j|^2+\cq u_j\cdot\nabla a_j\right)\dive u\de x\label{no higher terms}\\
&+\nu\int_{\R^d} (\nabla g_j-\nabla(u+V)\cdot\nabla a_j)\cdot\nabla a_j\de x+\int_{\R^d} (\nabla g_j-\nabla(u+V)\cdot\nabla a_j)\cdot\cq u_j\de x+\int_{\R^d} f_j\cdot\nabla a_j\de x.\nonumber
\end{align}
Recall that our goal is to control the norms appearing in the functional $\mathcal{X}$. To this end, we need suitable estimates for the triple $(\cq u_j,a_j,\nu\nabla a_j)$. We therefore take 
$$
2\cdot\eqref{3.36}+2\cdot\eqref{3.37New}+\nu\cdot\eqref{no higher terms},
$$
obtaining
\begin{align}
   \frac{1}{2}\frac{\de}{\de t}\int_{\R^d}(|\nu\nabla a_j|^2+&2\nu\cq u_j\cdot\nabla a_j+2a_j^2+2|\cq u_j|^2)\de x+\nu\int_{\R^d} (|\nabla a_j|^2+|\nabla\cq u_j|^2)\de x\nonumber\\
   &-2\int_{\R^d} \dot\Delta_j((h+B)\cdot\nabla(h+B))\cdot\cq u_j\de x+\int_{\R^d} \dot\Delta_j\nabla(|B+h|^2)\cdot\cq u_j\de x\nonumber\\
   &-\nu\int_{\R^d}  \dot\Delta_j(\cq((B+h)\cdot\nabla(B+h))\cdot\nabla a_j\de x+\frac{\nu}{2}\int_{\R^d}  \dot\Delta_j\nabla(|B+h|^2)\cdot\nabla a_j \de x\nonumber\\
   =&\int_{\R^d} (2 g_j a_j+2f_j\cdot\cq u_j+\nu^2\nabla g_j\cdot\nabla a_j+\nu \nabla g_j\cdot\cq u_j+\nu f_j\cdot \nabla a_j)\de x\nonumber\\
   &+\frac{1}{2}\int_{\R^d} (2a_j^2+2|\cq u_j|^2+2\nu \cq u_j\cdot\nabla a_j+|\nu \nabla a_j|^2)\dive u\de x\nonumber\\
   &-\nu \int_{\R^d} (\nabla(u+V)\cdot\nabla a_j)\cdot(\nu\nabla a_j+\cq u_j)\de x.\label{lunghissima}
\end{align}
In order to simplify the presentation, we define the quantity
\begin{equation}\label{def L_j^2}
    \mathscr{L}_j^2 :=\int_{\R^d} (2a_j^2+2|\cq u_j|^2+2\nu\cq u_j\cdot\nabla a_j+|\nu\nabla a_j|^2)\de x,
\end{equation}
and observe that $\forall j\in \mathbb{Z}$ one has
\begin{equation}\label{3.42}
    \mathscr{L}_j \approx \|(\cq u_j, a_j, \nu \nabla a_j)\|_{L^2},
\end{equation}
and
\begin{equation}\label{stima dal basso Lj}
    \nu \int_{\R^d} (|\nabla\cq u_j|^2+|\nabla a_j|^2)\de x\geq c \text{ min}(\nu 2^{2j},\nu^{-1})\mathscr{L}_j^2.
\end{equation}
Thus, by using the definition \eqref{def L_j^2}, Holder inequality and the lower bound \eqref{stima dal basso Lj}, we can rewrite \eqref{lunghissima} as
\begin{align*}
    \frac{1}{2}\frac{\de}{\de t}&\mathscr{L}_j^2+c\text{ min}(\nu 2^{2j},\nu^{-1})\mathscr{L}_j^2
    \leq \left(\frac{1}{2}\|\dive u\|_{L^\infty}+c\|\nabla(u+V)\|_{L^\infty}\right)\mathscr{L}_j^2+c\|[g_j,f_j,\nu\nabla g_j]\|_{L^2}\mathscr{L}_j\\
    &+\left(2\|\dot\Delta_j((h+B)\cdot\nabla(h+B))\|_{L^2}+\|\dot\Delta_j\nabla(|B+h|^2)\|_{L^2}\right)\left(\|\cq u_j\|_{L^2}+\frac{\nu}{2}\|\nabla a_j\|_{L^2}\right)\\
    &\leq c\|(\nabla u,\nabla V)\|_{L^\infty} \mathscr{L}_j^2+\left(c\|[g_j,f_j,\nu\nabla g_j]\|_{L^2}+2\|\dot\Delta_j((h+B)\cdot\nabla(h+B))\|_{L^2}+\|\dot\Delta_j\nabla(|B+h|^2)\|_{L^2}\right)\mathscr{L}_j.
\end{align*}
Dividing by $\mathscr{L}_j$ and integrating in time, we finally get
\begin{equation}\label{3.44New}
\begin{aligned}
\mathscr{L}_j(t)+c\text{ min}(\nu 2^{2j},\nu^{-1})\int_0^t \mathscr{L}_j\de \tau&\leq \mathscr{L}_j(0)+c\int_0^t \|\nabla(u+V)\|_{L^\infty}\mathscr{L}_j\de \tau+\int_0^t \|\dot\Delta_j((h+B)\cdot\nabla(h+B)\|_{L^2}\de \tau\\
&+\int_0^t \left(\|\dot\Delta_j\nabla(|B+h|^2)\|_{L^2}+\|[g_j,f_j,\nu\nabla g_j]\|_{L^2}\right)\de \tau
\end{aligned}
\end{equation}
We emphasize that, for high frequencies $j$, the relation $\min(\nu^{2j}, \nu^{-1}) = \nu^{-1}$ leads to a loss of parabolic smoothing for $\cq u$ in the above estimate. This difficulty can be overcome by returning to \eqref{3.37New} and performing an integration by parts in the term $a_j\,\dive \cq u_j$. Using Hölder’s and Bernstein’s inequalities then we get
\begin{align*}
    \frac{1}{2}\frac{\de}{\de t}\|\cq u_j\|_{L^2}^2+c\nu 2^{2j}&\|\cq u_j\|_{L^2}^2\leq \|\nabla a_j\|_{L^2}\|\cq u_j\|_{L^2}+\|\dot\Delta_j((h+B)\cdot\nabla(h+B))\|_{L^2}\|\cq u_j\|_{L^2}\\
    &+\frac{1}{2}\|\dot\Delta_j\nabla(|B+h|^2)\|_{L^2}\|\cq u_j\|_{L^2}+\frac{1}{2}\|\dive u\|_{L^\infty}\|\cq u_j\|_{L^2}^2+\|f_j\|_{L^2}\|\cq u_j\|_{L^2}.
\end{align*}
Therefore, we divide by $\|\cq u_j\|_{L^2}$, we integrate in time,
and we recover the Besov norm as usual to get
\begin{equation}\label{est cq u, nabla cq u}
    \begin{aligned}
         \| \cq u\|_{\dot B_{2,1}^{d/2-1}}+\nu \int_0^t \|\nabla\cq u\|_{\dot B_{2,1}^{d/2}}\de\tau\lesssim & \|\cq u(0)\|_{\dot B_{2,1}^{d/2-1} }+\nu^{-1}\int_0^t \|\nu\nabla a\|_{\dot B_{2,1}^{d/2-1}}\de\tau \\
    &+\int_0^t(\|(h+B)\cdot\nabla(h+B)\|_{\dot B_{2,1}^{d/2-1}}+\|\nabla(|h+B|^2)\|_{\dot B_{2,1}^{d/2-1}})\de\tau\\
    &+\int_0^t \|\nabla u\|_{L^\infty}\|\cq u\|_{\dot B_{2,1}^{d/2-1}}\de\tau+\int_0^t \sum_{j\in\mathbb{Z}} 2^{j(d/2-1)}\|f_j\|_{L^{2}}\de\tau,
    \end{aligned}
\end{equation}
where we have used Bernstein inequality in Lemma \ref{lem:bernstein} to bound from below the second term on the left-hand side as follows
\begin{equation*}
    c\nu\int_0^t\sum_{j\in\mathbb{Z}}2^{2j}2^{j(d/2-1)}\|\cq u_j\|_{L^2}\de\tau=c\nu\int_0^t \|\cq u\|_{\dot B_{2,1}^{d/2+1} }\de\tau\gtrsim \nu\int_0^t\|\nabla\cq u\|_{\dot B_{2,1}^{d/2}}\de\tau.
\end{equation*}
We now need to estimate separately the high- and low-frequency parts of $a$, as defined in \eqref{def:split frequenze}. First, we consider the low-frequency part $a^\ell$ and we get
\begin{align*}
    \nu \|\nabla a^\ell\|_{\dot B_{2,1}^{d/2}}&=\nu\sum_{2^j\nu\leq1 }2^{jd/2 }\|\dot\Delta_j \nabla a\|_{L^{2}}=\sum_{2^j\nu\leq1 } \nu 2^j2^{j(d/2-1) }\|\dot\Delta_j \nabla a\|_{L^{2}}\leq\sum_{j\in \mathbb{Z} }2^{j(d/2-1) }\|\nabla a_j\|_{L^2},
\end{align*}
where in the last inequality we used that we sum over frequencies $j$ such that $\nu 2^j\leq 1$.
On the other hand, for the high-frequency part $a^h$ we have
\begin{align*}
    \|a^h\|_{\dot B_{2,1}^{d/2} }&=\sum_{2^j\nu>1 }2^{jd/2 }\|\dot\Delta_j  a\|_{L^{2}}\overset{\eqref{Reverse Bernstein} }{\lesssim}\sum_{2^j\nu>1 }2^{jd/2 }2^{-j}\|\dot\Delta_j\nabla a\|_{L^{2}}.
\end{align*}
Then, we combine \eqref{3.44New}, \eqref{3.42} and \eqref{est cq u, nabla cq u} to obtain the estimate
\begin{align}
    \|(a,\nu\nabla a,\cq u)(t)\|_{\dot B_{2,1}^{d/2-1} }&+\nu\int_0^t \|\nabla a^\ell\|_{\dot B_{2,1}^{d/2} }\de\tau+\int_0^t \|\nabla a^h\|_{\dot B_{2,1}^{d/2} }\de\tau
    +\nu\int_0^t \|\nabla \cq u\|_{\dot B_{2,1}^{d/2} }\de\tau\nonumber\\
    \lesssim & \|(a,\nu\nabla a,\cq u)(t)\|_{\dot B_{2,1}^{d/2-1}}+\nu^{-1}\int_0^t \sum_{j\in\mathbb{Z}}2^{j(d/2-1)}\|\nu\nabla a_j\|_{L^2}\de\tau+\int_0^t \nu\|\nabla\cq u\|_{\dot B_{2,1}^{d/2} }\de\tau\nonumber\\
    &=\|(a,\nu\nabla a,\cq u)(t)\|_{\dot B_{2,1}^{d/2-1}}+\nu^{-1}\int_0^t\|\nu\nabla a\|_{\dot B_{2,1}^{d/2-1}}\de\tau+\nu\int_0^t \|\nabla\cq u\|_{\dot B_{2,1}^{d/2}}\de\tau\nonumber\\
    \lesssim& \|(a,\nu\nabla a,\cq u)(0)\|_{\dot B_{2,1}^{d/2-1}}
    +\int_0^t \|(\nabla u,\nabla V)\|_{L^\infty}\|(\cq u,a,\nu\nabla a)\|_{\dot B_{2,1}^{d/2-1} }\de\tau\nonumber\\
    &+\int_0^t ( \|(h+B)\cdot\nabla(h+B)\|_{\dot B_{2,1}^{d/2-1} }+\|\nabla(|h+B|^2)\|_{\dot B_{2,1}^{d/2-1} })\de\tau\nonumber
    \\
    &+\int_0^t \sum_{j\in\mathbb{Z}}2^{j(d/2-1)}\|[g_j,f_j,\nu\nabla g_j]\|_{L^2}\de\tau
    +\nu^{-1}\int_0^t \|\nu\nabla a\|_{\dot B_{2,1}^{d/2-1}}\de\tau.\label{3.45New}
\end{align}
Now, since $\nu$ is large, we can absorb the term $\nu^{-1}\displaystyle\int_0^t \|\nu\nabla a\|_{\dot B_{2,1}^{d/2-1}}\de \tau$ into the left-hand side. 
By Lemma \eqref{lemmastimaprod} and the commutator estimate \eqref{commutatorestimL2} we can estimate $g_j$ as
\begin{equation}\label{stima g_j}
    \begin{aligned}
        \int_0^t \sum_{j\in\Z} 2^{j(d/2-1)}\|g_j\|_{L^2}\de \tau\lesssim& \int_0^t \|\dive \cq u\|_{\dot B_{2,1}^{d/2}}\|a\|_{\dot B_{2,1}^{d/2-1}}\de\tau+\int_0^t \|\nabla(u+V)\|_{\dot B_{2,1}^{d/2}}\|a\|_{\dot B_{2,1}^{d/2}}\de\tau\\
        \lesssim& \int_0^t \|\nabla(u+V)\|_{\dot B_{2,1}^{d/2}}\left(\|a\|_{\dot B_{2,1}^{d/2-1}}+\|a\|_{\dot B_{2,1}^{d/2}} \right)\de\tau.
    \end{aligned}
\end{equation}
In the same spirit, the term involving $\nabla g_j$ can be bounded by means of the Leibniz rule as follows
\begin{equation}\label{stima nabla g_j}
        \int_0^t \sum_{j\in\Z}2^{j(d/2-1)}\|\nu\nabla g_j\|_{L^2}\de\tau\lesssim 
       \int_0^t \|\nabla(u+V)\|_{\dot B_{2,1}^{d/2}}\nu\|a\|_{\dot B_{2,1}^{d/2}}. 
\end{equation}
By similar arguments, and using the definition of $f_j$ in \eqref{def:fj}, we have
\begin{equation}\label{3.56}
    \begin{aligned}
         \sum_{j\in\mathbb{Z}}2^{j(d/2-1)}\|f_j\|_{L^2}\lesssim& \|(u,V)\|_{\dot B_{2,1}^{d/2-1}}\|(\nabla\cp u,\nabla V)\|_{\dot B_{2,1}^{d/2}}+\|a\|_{\dot B_{2,1}^{d/2}}(\|(u,V)\|_{\dot B_{2,1}^{d/2-1}}\|(\nabla u,\nabla V)\|_{\dot B_{2,1}^{d/2}}\\
         &+\|(V_t,\cp u_t,\cq u_t+\nabla a)\|_{\dot B_{2,1}^{d/2-1}}+\|a\|_{\dot B_{2,1}^{d/2}}).
    \end{aligned}
\end{equation}
Plugging estimates \eqref{stima g_j}, \eqref{stima nabla g_j} and \eqref{3.56} into \eqref{3.45New} and using the embedding $\dot B_{2,1}^{d/2}(\mathbb{R}^d)\hookrightarrow L^\infty (\mathbb{R}^d)$, for all $0<T<T_*$ we have that 
\begin{equation}\label{est X+Y senza Qut+nabla a}
\begin{aligned}
\|(a,\nu\nabla a,\cq u)&\|_{L^\infty\dot B_{2,1}^{d/2-1}}+\nu \int_0^T \|\nabla^2 a^\ell,\nabla^2\cq u\|_{\dot B_{2,1}^{d/2-1}}\de\tau +\int_0^T \|\nabla a^h\|_{\dot B_{2,1}^{d/2-1}} \de\tau\\
\lesssim & \|(a,\nu\nabla a,\cq u)(0)\|_{\dot B_{2,1}^{d/2-1}}+\int_0^T \|(\nabla u,\nabla V)\|_{\dot B_{2,1}^{d/2} }\|(\cq u,a,\nu\nabla a)\|_{\dot B_{2,1}^{d/2-1}}\de\tau\\
&+\left(1+\|a\|_{L^\infty\dot B_{2,1}^{d/2} }\right)\int_0^ T \|(\cp u,\cq u,V)\|_{\dot B_{2,1}^{d/2-1}}\|(\nabla\cp u,\nabla V)\|_{\dot B_{2,1}^{d/2} }\de\tau\\
&+\left(1+\|a\|_{L^\infty\dot B_{2,1}^{d/2} }\right)\int_0^ T \|(\cp u,\cq u,V)\|_{\dot B_{2,1}^{d/2-1}}\|\nabla \cq u\|_{\dot B_{2,1}^{d/2} }\de\tau\\
&+\|(V_t,\cp u_t,\cq u_t+\nabla a)\|_{L^1\dot B_{2,1}^{d/2-1} }\|a\|_{L^\infty\dot B_{2,1}^{d/2} }+\|a\|_{L^2\dot B_{2,1}^{d/2}}^2\\
&+\int_0^T \left(\|(h+B)\cdot\nabla(h+B)\|_{\dot B_{2,1}^{d/2-1}}+\|\nabla(|h+B|^2)\|_{\dot B_{2,1}^{d/2-1}}\right)\de\tau.
\end{aligned}
\end{equation}
We now need an estimate on $\cq u_t+\nabla a$. To do this, we simply use the equation \eqref{3.17NEW} to infer that
\begin{align*}
\|\cq u_t+\nabla a\|_{L^1\dot B_{2,1}^{d/2-1} }\leq& \nu \|\Delta\cq u\|_{L^1\dot B_{2,1}^{d/2-1}}+\|(u+V)\cdot\nabla\cq u\|_{L^1\dot B_{2,1}^{d/2-1} }\\
&+\|(h+B)\cdot\nabla(h+B)\|_{L^1\dot B_{2,1}^{d/2-1} }+\frac{1}{2}\|\nabla(|B+h|^2)\|_{L^1\dot B_{2,1}^{d/2-1} }+\|R_2\|_{L^1\dot B_{2,1}^{d/2-1} }\\
&+\|au_t+aV_t\|_{L^1\dot B_{2,1}^{d/2-1} }.
\end{align*}
Now, since
\begin{align}
    \|au_t+aV_t\|_{L^1\dot B_{2,1}^{d/2-1} }
    &\lesssim \|a\|_{L^\infty \dot B_{2,1}^{d/2}} \|(\cp u_t,V_t,\cq u_t+\nabla a)\|_{L^1(0,T,\dot B_{2,1}^{d/2-1})}+\|a\|_{L^2\dot B_{2,1}^{d/2}}^2,\label{stima au+aV}
\end{align}
and 
\begin{equation*}
    \begin{aligned}
        \int_0^T \|(u+V)\cdot\nabla\cq u\|_{\dot B_{2,1}^{d/2-1} }\de t
        \lesssim& \int_0^T \|(\nabla\cp u,\nabla\cq u,\nabla V)\|_{\dot B_{2,1}^{d/2} }\|\cq u\|_{\dot B_{2,1}^{d/2-1} }   \de t.
    \end{aligned}
\end{equation*}
we can combine the estimate \eqref{stima au+aV} with \eqref{est X+Y senza Qut+nabla a}, to get

\begin{equation}\label{3.57New}
\begin{aligned}
\|(a,\nu\nabla a,\cq u)\|_{L^\infty\dot B_{2,1}^{d/2-1}}&+\|\cq u_t+\nabla a,\nu\nabla^2\cq u,\nu\nabla^2 a^\ell,\nabla a^h\|_{L^1\dot B_{2,1}^{d/2-1}}\\
\lesssim& \|(a,\nu\nabla a,\cq u)(0)\|_{\dot B_{2,1}^{d/2-1}}+\int_0^T \|\nabla u,\nabla V\|_{\dot B_{2,1}^{d/2} }\|(a,\nu\nabla a,\cq u)\|_{\dot B_{2,1}^{d/2-1}}\de\tau\\
&+\left(1+\|a\|_{L^\infty B_{2,1}^{d/2}}\right)\int_0^T \|(\cp u,\cq u,V)\|_{\dot B_{2,1}^{d/2-1}}\|\nabla\cp u,\nabla V\|_{\dot B_{2,1}^{d/2} }\de\tau\\
&+\left(1+\|a\|_{L^\infty B_{2,1}^{d/2}}\right)\int_0^T \|(\cp u,\cq u,V)\|_{\dot B_{2,1}^{d/2-1}}\|\nabla \cq u\|_{\dot B_{2,1}^{d/2} }\de\tau\\
&+\|(V_t,\cp u_t,\cq u_t+\nabla a)\|_{L^1\dot B_{2,1}^{d/2-1}}\|a\|_{L^\infty\dot B_{2,1}^{d/2}}+\|a\|_{L^2\dot B_{2,1}^{d/2}}^2\\
&+\int_0^T \|(\nabla\cp u,\nabla\cq u,\nabla V)\|_{\dot B_{2,1}^{d/2} }\|\cq u\|_{\dot B_{2,1}^{d/2-1} }   \de t+\|(h+B)\cdot\nabla(h+B)\|_{L^1\dot B_{2,1}^{d/2-1}}\\
& +\|\nabla(|h+B|^2)\|_{L^1\dot B_{2,1}^{d/2-1}}+\|R_2\|_{L^1\dot B_{2,1}^{d/2-1}}. 
\end{aligned}
\end{equation}
Then, assumining that $\delta$ is small enough and using the fact that \eqref{3.29} implies that
\begin{equation}\label{3.58}
    \|a\|_{L^\infty\dot B_{2,1}^{d/2}}\ll 1,
\end{equation}
the estimate in \eqref{3.57New} reduces to
\begin{align*}
\|(a,\nu\nabla a,\cq u)&\|_{L^\infty\dot B_{2,1}^{d/2-1}}+\|\cq u_t+\nabla a,\nu\nabla^2\cq u,\nu\nabla^2 a^\ell,\nabla a^h\|_{L^1\dot B_{2,1}^{d/2-1}}\\
\lesssim& \|(a,\nu\nabla a,\cq u)(0)\|_{\dot B_{2,1}^{d/2-1}}+\int_0^T \|(\nabla \cp u,\nabla \cq u,\nabla V)\|_{\dot B_{2,1}^{d/2} }\|(\cq u,a,\nu\nabla a)\|_{\dot B_{2,1}^{d/2-1}}\de\tau\\
& +\int_0^T \|V\|_{\dot B_{2,1}^{d/2-1} }^2\|(\cq u,a,\nu\nabla a)\|_{\dot B_{2,1}^{d/2-1}}\de\tau+\|(\cp u,V)\|_{L^\infty \dot B_{2,1}^{d/2-1}}\|(\nabla\cp u,\nabla V)\|_{L^1 \dot B_{2,1}^{d/2}}\\
&+\|a\|_{L^\infty\dot B_{2,1}^{d/2}} \|(\cp u,V)\|_{L^\infty \dot B_{2,1}^{d/2-1}}\|(\nabla\cp u,\nabla V)\|_{L^1\dot B_{2,1}^{d/2} }\\
&+\|(V_t,\cp u_t)\|_{L^1\dot B_{2,1}^{d/2-1}}\|a\|_{L^\infty \dot B_{2,1}^{d/2} } +\|(h+B)\cdot\nabla(h+B)\|_{L^1 \dot B_{2,1}^{d/2-1} }+\|\nabla(|B+h|^2)\|_{L^1 \dot B_{2,1}^{d/2-1} }\\
&+\nu^{-1}\left(\|a^\ell\|_{L^\infty \dot B_{2,1}^{d/2-1}}\|\nu\nabla a^\ell\|_{L^1\dot B_{2,1}^{d/2}}+\|\nu a^h\|_{L^\infty \dot B_{2,1}^{d/2} }\|a^h\|_{L^1\dot B_{2,1}^{d/2} }\right)
\end{align*}
In order to conclude, we need to handle the terms involving the magnetic field. As already done before, we use Lemma \ref{lemmastimaprod} and the definition of the quantities in \eqref{def:quantità}, to get that
\begin{align}
    \|(h+B)\cdot\nabla(h+B)\|_{L^1\dot B_{2,1}^{d/2-1}}
    \leq& \|h\|_{L^\infty \dot B_{2,1}^{d/2-1}}\|B\|_{L^1\dot B_{2,1}^{d/2+1}}+\|B\|_{L^\infty \dot B_{2,1}^{d/2-1}}\|\nabla^2 h\|_{L^1\dot B_{2,1}^{d/2-1}}\nonumber\\
    &+\|h\|_{L^\infty \dot B_{2,1}^{d/2-1}}\|\nabla^2 h\|_{L^1\dot B_{2,1}^{d/2-1}}+\|B\|_{L^\infty \dot B_{2,1}^{d/2-1} }\| B\|_{L^1\dot B_{2,1}^{d/2+1}}\nonumber\\
    \lesssim& \mathcal{Z}\mathcal{V}+\mathcal{W}\mathcal{V}+\mathcal{Z}\mathcal{W}+\mathcal{V}^2,
\end{align}
Moreover, since 
\begin{align*}
\int_0^T \|\nabla(B\cdot h)\|_{\dot B_{2,1}^{d/2-1}}\de\tau 
\lesssim&  \| B\|_{L^1\dot B_{2,1}^{d/2+1}}\|h\|_{L^\infty \dot B_{2,1}^{d/2-1}}+ \|\nabla^2 h\|_{L^1\dot B_{2,1}^{d/2-1}}\|B\|_{L^\infty \dot B_{2,1}^{d/2-1}}\\
\lesssim& \mathcal{V}\mathcal{Z}+\mathcal{V}\mathcal{W},
\end{align*}
we obtain the following bound for the pressure-like term $\nabla (|B+h|^2)$
\begin{align*}
\|\nabla(|B+h|^2)\|_{L^1\dot B_{2,1}^{d/2-1}}
\lesssim& \|\nabla |B|^2\|_{L^1\dot B_{2,1}^{d/2-1}}+\|\nabla |h|^2 \|_{L^1\dot B_{2,1}^{d/2-1}}+\|\nabla(B\cdot h)\|_{L^1\dot B_{2,1}^{d/2-1}}\\
\lesssim& \mathcal{V}^2+\mathcal{Z}\mathcal{W}+\mathcal{V}\mathcal{Z}+\mathcal{V}\mathcal{W},
\end{align*}
Hence, we can rewrite \eqref{3.57New}, in terms of the quantities defined in \eqref{def:quantità} as
\begin{align}
\begin{aligned}
\mathcal{X}+\mathcal{Y}
&\lesssim
\mathcal{X}(0)+\int_0^T
\Bigl(
\|\nabla\cp u,\nabla\cq u,\nabla V\|_{L^1\dot B_{2,1}^{d/2} }+\|V\|^2_{L^2\dot B_{2,1}^{d/2} }\Bigr)
\bigl(\mathcal{X}(\tau)+\mathcal{Y}(\tau)\bigr)\,\de \tau
\\
&+ \int_0^T
\Bigl[
(\mathcal{V}(\tau)+\mathcal{Z}(\tau))(\mathcal{V}(\tau)+\mathcal{W}(\tau))\Bigr]\de \tau\\
&+\int_0^T
\Bigl[\nu^{-2}\mathcal{X}(\tau)\mathcal{Y}(\tau)(\mathcal{V}(\tau)+\mathcal{Z}(\tau))
+\nu^{-1}(\mathcal{V}(\tau)+\mathcal{Y}(\tau)+\mathcal{W}(\tau))\mathcal{X}(\tau)
\Bigr]\,\de\tau .
\end{aligned}
\end{align}
and then from Gronwall lemma, we obtain the following bound for the potential part
\begin{equation}\label{bound Compr part}
    \begin{aligned}
        \mathcal{X}+&\mathcal{Y}\leq Ce^{C\|\nabla\cp u,\nabla\cq u,\nabla V\|_{L^1\dot B_{2,1}^{d/2} }+\|V\|^2_{L^2\dot B_{2,1}^{d/2} }}(X_d(0)
        +(\mathcal{V}+\mathcal{Z})(\mathcal{V}+\mathcal{W})\\
        &+\nu^{-2}\mathcal{X}\mathcal{Y}(\mathcal{V}+\mathcal{Z})+\nu^{-1}(\mathcal{V}+\mathcal{Y}+\mathcal{W})\mathcal{X}).
    \end{aligned}
\end{equation}
{\em \underline{Step 4} \quad Closure of the estimates.}\\
\\
In this final step, we collect all the estimates obtained so far and choose the parameters appropriately in order to close the argument.
First of all, we assume that
\begin{equation}\label{D lessless nu}
    \nu^{-1}D\ll 1.
\end{equation}
Then, for the compressible part, namely to estimate $\mathcal{X}+\mathcal{Y}$, we can proceed as follows. We use the interpolation inequality \eqref{interpolation} to bound
$$
\int_0^T \|V\|_{\dot B_{2,1}^{d/2}}^2\de\tau \lesssim\|V\|_{L^\infty \dot B_{2,1}^{d/2-1}}\|\nabla^2V\|_{L^1 \dot B_{2,1}^{d/2-1}}.
$$
Therefore, using inequality \eqref{bound Compr part} together with assumptions \eqref{bound V} and \eqref{local bounds X Y Z W}, we obtain
\begin{equation}
\label{X+Y quasi finito}
\mathcal{X}+\mathcal{Y}\leq Ce^{C(M+M^2+\nu^{-1}D+\delta)}(\mathcal{X}(0)+(M+\delta)^2+\nu^{-2}D(M+\delta)\mathcal{X}+\nu^{-1}(D+\delta+M)\mathcal{X}).
\end{equation}
On the other hand, to estimate the incompressible part $\mathcal{Z}+\mathcal{W}$ in \eqref{local est Z+W}, we use \eqref{interpolation}, \eqref{def:quantità}, and \eqref{local bounds X Y Z W} to get 
\begin{align*}
\|\cq u\|^2_{L^2\dot B_{d,1}^{d/2}}\leq \|\cq u\|_{L^\infty\dot B_{d,1}^{d/2-1}}\|\nabla\cq u\|_{L^1\dot B_{d,1}^{d/2}}\leq \nu^{-1}D^2,
\end{align*}
and finally we obtain
\begin{align}
\mathcal{Z}+\mathcal{W}
\lesssim& DCe^{C(M+\delta+\nu^{-1}D+\nu^{-1}D^2+M^2)}(\nu^{-1}(D+\delta+M)+\nu^{-1/2}M+\nu^{-1}(\delta+D+M)\mathcal{W}\nonumber\\
&+\nu^{-1}M^2+\nu^{-1}M).\label{Z+W quasi finito}
\end{align}
We assume that 
\begin{equation}\label{relazione nu, D, M}
\nu^{-1}D+\nu^{-1}D^2\leq M+M^2 \quad \text{and} \quad \delta \leq \text{max}\{M,1\},
\end{equation}
hence \eqref{X+Y quasi finito} and \eqref{Z+W quasi finito} become
\begin{align*}
    \mathcal{X}+\mathcal{Y}&\leq Ce^{C(M+M^2) }(\mathcal{X}(0)+M^2+1+\nu^{-1}(M+D+1)\mathcal{X}),\\
\mathcal{Z}+\mathcal{W}  
&\leq DCe^{C(M+M^2)}(\nu^{-1}D+\nu^{-1/2}M+\nu^{-1}M^2+\nu^{-1}(M+D+1)\mathcal{W}).
\end{align*}
Therefore, we consider $\nu$ big enough so that
\begin{equation}\label{3.62}
    D(D+M+1)e^{C(M+M^2)}\ll\nu,
\end{equation}
in order to have the bounds
\begin{equation}\label{X+Y}
\mathcal{X}+\mathcal{Y}\leq Ce^{C(M+M^2)}(\mathcal{X}(0)+M^2+1),
\end{equation}
and
\begin{equation}\label{Z+W}
    \begin{aligned}
    \mathcal{Z}+\mathcal{W}
    \leq& CDe^{C(M+M^2)}(\nu^{-1/2}M+\nu^{-1}(D+M^2+D^2+1)),
    \end{aligned}
\end{equation}
where, in the last inequality we used Young's inequality
\begin{equation*}
\nu^{-1}DM
\lesssim \nu^{-1}M^2+\nu^{-1}D^2.
\end{equation*}
We can now explicitly define the parameters $D$ and $\delta$: we set
\begin{equation}\label{def D}
    D :=Ce^{C(M+M^2)}(\mathcal{X}(0)+M^2+1),
\end{equation}
and 
\begin{equation}\label{def delta}
    \delta :=Ce^{2C(M+M^2)}D\left(\nu^{-1/2}M+\nu^{-1}(\mathcal{X}(0)+M^2+1)+\nu^{-1}(\mathcal{X}(0)+M^2+1)^2\right).
\end{equation}
Thus, if $C$ and $\nu$ are sufficiently large and satisfy the bound
\begin{equation}
    Ce^{C(M+M^2)}(\mathcal{X}(0)+M^2+1)\leq \sqrt{\nu},
\end{equation}
the \eqref{D lessless nu} and \eqref{3.62} are automatically satisfied. Indeed
\begin{equation*}
    \nu^{-1}D\leq \nu^{-1}\sqrt{\nu}=\nu^{-1/2}\ll1,
\end{equation*}
and
\begin{equation*}
    D(D+M)e^{C(M+M^2)}  \leq \sqrt{\nu}\sqrt{\nu}+\sqrt{\nu}\ll 2\nu.
\end{equation*}
Then, by choosing $D$ and $\delta$ as in \eqref{def D} and \eqref{def delta}, the inequalities in \eqref{X+Y} and \eqref{Z+W} implies that the conditions \eqref{local bounds X Y Z W} (and consequently \eqref{3.29}) hold for all $T<T_*$. Moreover, by definition of $\mathcal{X}$, \eqref{local bounds X Y Z W} and \eqref{D lessless nu} we have for all $t\in[0,T]$
\begin{equation}
\|a(t)\|_{L^\infty}\lesssim \|a(t)\|_{\dot B_{2,1}^{d/2} }\lesssim \nu^{-1}\mathcal{X}\lesssim \nu^{-1}D\ll1.
\end{equation}
Therefore, the assumptions of the continuation criterion in Theorem \ref{thm:continuation} are satisfied, which yields the existence of a global strong solution to \eqref{MHDcomp}.
Finally, if $a_0=0,$ then the above analysis implies that
\begin{equation*}
    \begin{aligned}
        &\nu\|\rho-1\|_{L^\infty(\R_+;\dot B_{2,1}^{d/2})}\lesssim \nu \|\nabla a\|_{L^\infty(\R_+;\dot B_{2,1}^{d/2-1})}\leq D,\\
        &\nu \|\nabla^2\cq v\|_{L^1(\R_+;\dot B_{2,1}^{d/2})}\leq D,\\
        &\|\cp v-V\|_{L^\infty(\R_+;\dot B_{2,1}^{d/2-1}) }\lesssim\|u\|_{L^\infty(\R_+;\dot B_{2,1}^{d/2-1}) }\leq C\nu^{-1/2},\\
        &\|v-V_t,\mu\nabla^2(\cp v-V)\|_{L^1(\R_+;\dot B_{2,1}^{d/2-1})}\leq C\nu^{-1/2},\\
        &\|b-B\|_{L^\infty(\R_+;\dot B_{2,1}^{d/2-1})}\leq C\nu^{-1/2},\\
        &\|b_t-B_t,\eta\nabla^2(b-B)\|_{L^1(\R_+;\dot B_{2,1}^{d/2-1})}\leq C\nu^{-1/2}.
    \end{aligned}
\end{equation*}  
Therefore, one can show that the solution of the compressible system converges to the incompressible one with the quantitative rate of convergence as in \eqref{estconv}. Finally, the quantitative bound on the constan $M$ follows by Lemma \ref{lem:M}.
\end{proof}
\section{\texorpdfstring{$L^p$}{Lp} Besov framework}
\label{lp framework}
In this section, we extend the result of Theorem \ref{thm:main} to $L^p$ based Besov space, proving Theorem \ref{thm:Lp intro}. 
The strategy developed in the $L^2$ framework carries over to the $L^p$ case, although several steps require more delicate arguments and sharper estimates. 
For the reader’s convenience, we recall here the statement of the theorem in the $L^p$ context before turning to its proof.
\MainTheoremB*
\begin{proof}
We divide the proof in several steps. In what follows, we briefly indicate the arguments that are analogous to the $L^2$ case and focus on the modifications required by the $L^p$ framework. Moreover, we do not reproduce all the estimates in full detail. On the one hand, several arguments are straightforward adaptations of those used in the proof of Theorem \ref{thm:main}. On the other hand, in the $L^p$ framework the main difficulty concerns the compressible part of the dynamics. Since the magnetic field is divergence-free, it does not significantly affect this analysis, and the corresponding estimates can be carried out exactly as in the compressible Navier–Stokes case treated in \cite{Chen-Zhai:Lp}.\\
\\
{\em \underline{Step 0} \quad Local well-posedness.}\\
\\
First of all, we note that from the embedding $\dot B_{2,1}^{d/2}\hookrightarrow \dot B_{p,1}^{d/p}$ and using that $\nu^{-1}\ll1$, we have 
$$
\|a_0\|_{\dot B_{p,1}^{d/p} }\lesssim \|a_0^h\|_{\dot B_{p,1}^{d/p}}+ \|a_0^\ell\|_{\dot B_{2,1}^{d/2}}\lesssim \|a_0^h\|_{\dot B_{p,1}^{d/p}}+ \|a_0^\ell\|_{\dot B_{2,1}^{d/2-1}}.
$$
The same estimate applies to the velocity field $v_0$. Thus, we can invoke Theorem \ref{thm:lwp}: there exists some $T>0$ and a unique solution $(\rho,v,b)$ to \eqref{MHDcomp} satisfying
\begin{equation*}
    \begin{aligned}
        &a=\rho-1\in C([0,T);\dot B_{p,1}^{d/p}),\\
        &v,b\in C([0,T),\dot B_{p,1}^{d/p-1})\cap L^1((0,T);\dot B_{p,1}^{d/p+1}).
    \end{aligned}
\end{equation*}
Then, we define the differences
$$
u:=v-V,\qquad h:=b-B,
$$
and we introduce the analogous energy functionals defined in Theorem \ref{thm:main}.
\begin{align*}
\xp:=&\|(a^\ell,\nu\nabla a^\ell,\cq u^\ell)\|_{L^\infty\dot B_{2,1}^{d/2-1} }+\|\nu a^h\|_{L^\infty\dot B_{p,1}^{d/p} }+\|\cq u^h\|_{L^\infty\dot B_{p,1}^{d/p-1}}, \\
\yp:=& \|\nu a^\ell,\nu^2\nabla a^\ell,\nu\cq u^\ell)\|_{L^1\dot B_{2,1}^{d/2+1}}+\|a^h\|_{L^1\dot B_{p,1}^{d/p}}+\|\nu\cq u^h\|_{L^1\dot B_{p,1}^{d/p+1}}\\
&+\|(\cq u_t+\nabla a)^\ell\|_{L^1\dot B_{2,1}^{d/2-1}}+\|(\cq u_t+\nabla a)^h\|_{L^1\dot B_{p,1}^{d/p-1}} \\
    \zp :=&\|\cp u\|_{L^\infty\dot B_{p,1}^{d/p-1}}+\|h\|_{L^\infty\dot B_{p,1}^{d/p-1}}, \\
    \wp :=& \|\cp u_t\|_{L^1\dot B_{p,1}^{d/p-1}}+\|\cp u\|_{L^1\dot B_{p,1}^{d/p+1}}+\|h_t\|_{L^1\dot B_{p,1}^{d/p-1}}+\|h\|_{L^1\dot B_{p,1}^{d/p+1}}\\
    \vp:=& \|V\|_{L^\infty\dot B_{p,1}^{d/p-1}}+\|V_t\|_{L^1\dot B_{p,1}^{d/p-1}}+\|V\|_{L^1\dot B_{p,1}^{d/p+1}}+\|B\|_{L^\infty\dot B_{p,1}^{d/p-1}}+\|B\|_{L^1\dot B_{p,1}^{d/p+1}}.
\end{align*}
Again, we point out that all the temporal norms are computed over the time interval $(0,T)$ in which the local solution is defined. We omit the domain dependece in the norms above to lighten the notation.
Our goal now is to prove that if $\nu$ is sufficiently large (with respect to suitable norms of the initial data and of the reference incompressible solution) then one can find some (large) $D$ and (small) $\delta$ so that for all $T<T^*$:
\begin{equation}\label{claim_xp+yp_zp+wp}
    \xp+\yp\leq D \quad \text{and} \quad \zp+\wp\leq \delta.
\end{equation}
\\
{\em \underline{Step 1} \quad Estimates for the Incompressible Part of \eqref{MHDcomp}.}\\
\\
In this step we look for bounds on the quantities $\zp$ and $\wp$.
To do that, we consider the evolution of the incompressible part of the system \eqref{Peq}, and we rewrite it as a coupled system of heat equations
\begin{equation}\label{parte inc del sist}
    \begin{cases}
        \cp u_t-\Delta \cp u=-\cp H_1, \\
        h_t-\Delta h=-\cp F_1,\\
        \cp u|_{t=0}=0, \ h|_{t=0}=0,
    \end{cases}
\end{equation}
with forcing terms defined as follows
\begin{align*}
    H_1:=& a(V_t+\cp u_t+(\cq u_t+\nabla a))+(1+a)\cp u\cdot\nabla(V+\cq u)+a(u+V)\cdot\nabla\cp u \\
    &+(u+V)\cdot\nabla\cp u+(1+a)(V\cdot\nabla\cq u+\cq u\cdot\nabla V)+a(\cq u\cdot\nabla\cq u+V\cdot\nabla V) \\
    &-h\cdot\nabla B-B\cdot\nabla h-h\cdot\nabla h, \\
    F_1:=&-\dive(\cq u)h-B\dive u-u\cdot\nabla B-\cp u\cdot\nabla h-\cq u\cdot\nabla h-V\cdot\nabla h+B\cdot\nabla u+h\cdot\nabla\cp u\\
    &+h\cdot\nabla\cq u+h\cdot\nabla V.
\end{align*}
The derivation of energy estimates for this system is largely parallel to the $L^2$ setting, up to suitable modifications accounting for the use of $L^p$–based Besov spaces. Indeed, we apply $\dot \Delta_j$ to both side of $\eqref{Peq}_1$, we multiply by $|\dot \Delta_j\cp u|^{p-2}\dot \Delta_j\cp u $ and we integrate in space, obtaining 
\begin{equation*}
    \frac{1}{p}\frac{\de}{\de t}\|\dot \Delta_j\cp u\|_{L^p}^p+C_12^{2j}\|\dot \Delta_j\cp u\|_{L^p}^p\lesssim \|\dot \Delta_j\cp H_1\|_{L^p}\|\dot \Delta_j\cp u\|_{L^p}^{p-1}.
\end{equation*}
Thus, we can divide by $\|\dot \Delta_j\cp u\|_{L^p}^{p-1}$, we multiply by $2^{j(d/p-1)}$, we sum over $j\in\Z  $ and we integrate over $(0,T)$ to get the estimate 
\begin{equation}\label{est:Pu Lp}
\|\cp u\|_{L_{T}^\infty \dot B_{p,1}^{d/p-1} }+\|\cp u\|_{L_{T}^1\dot B_{p,1}^{d/p+1}}\lesssim \int_0^T \|\cp H_1\|_{ \dot B_{p,1}^{d/p-1} }\de t.
\end{equation}
From now on, we keep track of the dependence of the space–time norms on the time interval $(0,T)$ as this will be needed to apply Gronwall’s inequality in a transparent way.
With a similar argument, for the magnetic field we also get that
\begin{equation}\label{{est:h Lp}}
    \|h\|_{L_{T}^\infty \dot B_{p,1}^{d/p-1} }+\|h\|_{L_{T}^1\dot B_{p,1}^{d/p+1}}\lesssim \int_0^T \|\cp F_1\|_{ \dot B_{p,1}^{d/p-1} }\de t.
\end{equation}
We now need to bound the remainder terms: as in the $L^2$ case, we use product laws Lemma \ref{lemmastimaprodLP} and commutator estimates in Lemma \ref{commut est Lpsetting}. Moreover, we can roughly estimate $\|\cp u_t\|_{L^1_T\dot B_{p,1}^{d/p-1} } $ and $\|h_t\|_{L^1_T\dot B_{p,1}^{d/p-1} }$ directly from the equations \eqref{parte inc del sist}, and we get that 
\begin{align*}
  &\bullet \|\cp H_1\|_{ \dot B_{p,1}^{d/p-1} }\lesssim  \nu^{-1}\xp(\yp+\wp+\vp)+\int_0^T \left(\|\nabla V\|_{\dot B_{p,1}^{d/p} }+\|\nabla \cq u\|_{\dot B_{p,1}^{d/p} } \right)\|\cp u\|_{\dot B_{p,1}^{d/p-1} }\de t\\
    &\hspace{3cm}+\nu^{-1}\xp(\zp+\xp+\vp)\wp+\varepsilon\|\cp u\|_{L^1(\dot B_{p,1}^{d/p+1}) } +(1+\nu^{-1}\xp)\nu^{-1/2}\xp^{1/2}\yp^{1/2}\vp\\
    &\hspace{3cm}+\int_0^T \left(\|u\|_{\dot B_{p,1}^{d/p}}^2+\|V\|_{\dot B_{p,1}^{d/p}}^2 \right)\|\cp u\|_{\dot B_{p,1}^{d/p-1} }\de t+\nu^{-1}\xp(\nu^{-1}\xp\yp+\vp^2)\\
    &\hspace{3cm}
    +\int_0^T \left(\|B\|_{\dot B_{p,1}^{d/p+1}  }+\|h\|_{\dot B_{p,1}^{d/p+1}}\right) \|h\|_{\dot B_{p,1}^{d/p-1} }\de t, \\
  &\bullet \|\cp F_1\|_{ \dot B_{p,1}^{d/p-1} }\lesssim 
  \int_0^T (\|\nabla\cq u\|_{\dot B_{p,1}^{d/p} }
  +\|\nabla V\|_{\dot B_{p,1}^{d/p}}
  +\|\cp u\|_{\dot B_{p,1}^{d/p+1} }
  )\|h\|_{\dot B_{p,1}^{d/p-1} }\de t+\nu^{-1}\vp\yp
  \\
  &\hspace{3cm} +\int_0^T( \|B\|_{\dot B_{p,1}^{d/p+1} }
  +\|h\|_{\dot B_{p,1}^{d/p+1}} )\|\cp u\|_{\dot B_{p,1}^{d/p-1} }\de t,\\
  &\bullet \|\cp u_t\|_{L^1_T\dot B_{p,1}^{d/p-1} }\lesssim \|\cp u\|_{L^1_T\dot B_{p,1}^{d/p+1} }+ \|\cp H_1\|_{L^1_T\dot B_{p,1}^{d/p-1}}, \\
   &\bullet \|\cp h_t\|_{L^1_T\dot B_{p,1}^{d/p-1}}\lesssim  
   \|h\|_{L^1_T\dot B_{p,1}^{d/p+1}}+\|\cp F_1\|_{L^1_T\dot B_{p,1}^{d/p-1}}.
\end{align*}

So, by defining
\begin{align}
    f(t)&=\|(\nabla V,\nabla\cq u)\|_{\dot B_{p,1}^{d/p} }+\|u\|^2_{\dot B_{p,1}^{d/p} }+\|V\|_{\dot B_{p,1}^{d/p} }^2+\|B\|_{\dot B_{p,1}^{d/p+1}}+\|h\|_{\dot B_{p,1}^{d/p+1}  }+\|\cp u\|_{\dot B_{p,1}^{d/p+1}},\\
    g(t)&=\nu^{-1/2}\xp^{1/2}\yp^{1/2}\vp+\nu^{-1}\xp(\zp+\xp+\vp)\wp+\nu^{-1}(\yp+\wp+\vp)\xp\nonumber\\
    &+\nu^{-1}\xp(\vp^2+\nu^{-1}\xp\yp)+\nu^{-1}\vp\yp
\end{align}
we can rewrite our estimate as 
\begin{equation*}
\zp+\wp \lesssim \int_0^T f(t)\left(\|\cp u\|_{\dot B_{p,1}^{d/p-1}} +\|h\|_{\dot B_{p,1}^{d/p-1}}\right)\de t+g(t),
\end{equation*}
and by Gr\"onwall lemma we can conclude that
\begin{equation}\label{est zp+wp}
\zp+\wp \leq Ce^{ C\int_0^Tf(t)\de t}g(t).
\end{equation}
\\
{\em \underline{Step 2} \quad High-Frequency Estimates for the Compressible Part of \eqref{MHDcomp}.}\\
\\
We now turn to the analysis of the compressible component, where the main differences with respect to the incompressible part arise. In this case, it is necessary to distinguish between low and high frequency regimes. We therefore consider the evolution of the compressible part \eqref{eq:differenze} and rewrite it in the following form
\begin{equation}\label{sist comp}
   \begin{cases}
       a_t+ (u+V)\cdot\nabla a =-a\dive (\cq u)-\dive (\cq u), \\
       \cq u_t-\nu\Delta\cq u+\nabla a=-\cq H_2, \\
       a|_{t=0}=a_0, \quad \cq u|_{t=0}=\cq v_0.
   \end{cases}
\end{equation}
where the remainder term $H_2$ is defined as
\begin{equation}
    \begin{aligned}
        H_2:=& a(V_t+\cp u_t+(\cq u_t+\nabla a))+(1+a)(u+V)\cdot\nabla(u+V)+(k(a)-a)\nabla a\\
        &-(h+B)\cdot\nabla(h+B))-\frac{1}{2}\nabla(|h+B|^2).
    \end{aligned}
\end{equation}
Note that the contribution of the magnetic field is relatively mild in \eqref{sist comp}. Indeed, the divergence-free condition prevents it from affecting the leading compressible dynamics, and its effect is therefore encoded only in the remainder terms, which will be treated perturbatively. 
On the other hand, the density equation is of transport type, while the equation for the compressible part of the velocity $\cq u$ involves $\nabla a$ as a source term. When trying to perform an $L^p$–energy estimate, these two equations cannot be combined, since the cancellations that arise from integration by parts in the $L^2$ framework are no longer available in $L^p$. Specifically, in the $L^2$ setting we first derive an energy balance for the mixed term $\int_{\R^d}\nabla a\cdot \cq u$ and then we find estimates on the corresponding norm. This strategy seems not to work in the $L^p$ case. To overcome this difficulty, one must introduce the effective velocity
\begin{equation}\label{effective velocity}
    w:=\cq u+\nu^{-1}(-\Delta)^{-1}\nabla a.
\end{equation}
which combines information from both the density and the velocity. This change of variables absorbs the gradient term $\nabla a$ appearing in the equation for $\cq u$, and transforms the system into one where $w$ satisfies a diffusive equation. By rewring the equation for the density in terms of $w$, the dissipation acting on $w$ can be transferred to $a$. Consequently, one can recover a damping effect for the density and, by combining the estimates for $w$ and $a$, obtain suitable bounds for $\cq u$. Thus, substituting \eqref{effective velocity} into \eqref{sist comp}, we obtain 
\begin{equation}\label{sist effective vel}
    \begin{cases}
        w_t-\nu\Delta w=\nu^{-1}w-\nu^{-2}(-\Delta)^{-1}\nabla a+\nu^{-1}\cq(a(u+V))-\cq H_2, \\
        a_t+(u+V)\cdot\nabla a+\nu^{-1}a=-a\dive u-\dive w.
    \end{cases}
\end{equation}
We now estimate $w$ by adapting the arguments developed for the incompressible component: we apply the operator $\dot\Delta_j$ to system \eqref{sist effective vel}, we multiply by $|\dot\Delta_j w|^{p-2}\dot\Delta_j w$, we integrate in space, we multiply by $2^{j(d/p-1)}$ and we sum over the high-frequency ($2^j\nu>1$) and we obtain 
\begin{equation}
    \begin{aligned}
            \|w^h\|_{L^\infty\dot B_{p,1}^{d/p-1} }+\nu\|w^h\|_{L^1\dot B_{p,1}^{d/p+1} }\lesssim &\|w_0^h\|_{\dot B_{p,1}^{d/p-1}}
            +\|\nu^{-1}w^h\|_{L^1\dot B_{p,1}^{d/p-1} }+\|\nu^{-2}(-\Delta)^{-1}\nabla a^h\|_{L^1\dot B_{p,1}^{d/p-1} }\\
            &+\|\nu^{-1}\cq(a(u+V))^h\|_{L^1\dot B_{p,1}^{d/p-1} }
            +\|\cq H_2^h\|_{L^1\dot B_{p,1}^{d/p-1} }.
    \end{aligned}
\end{equation}
By using the product estimate in Besov spaces in Lemma \ref{lemmastimaprodLP} and similar arguments to those above, we obtain
\begin{equation}\label{est effective velocity}
    \begin{aligned}
        \|w^h\|_{L^\infty\dot B_{p,1}^{d/p-1} }+\nu\|w^h\|_{L^1\dot B_{p,1}^{d/p+1} }\lesssim & \|w_0^h\|_{\dot B_{p,1}^{d/p-1}}
            +\|\nu^{-1}w^h\|_{L^1\dot B_{p,1}^{d/p-1} }+\nu^{-2}\|a^h\|_{\dot B_{p,1}^{d/p}}+(\vp+\wp)(\zp+\vp)\\
            &+\nu^{-1}\yp(\zp+\vp)+\nu^{-1}\xp(\wp+\vp+\yp)+\nu^{-1}\xp\yp\\
            &+\int_0^T \|(\nabla u,\nabla V)\|_{\dot B_{p,1}^{d/p} }\left(\|\cq u^h\|_{\dot B_{p,1}^{d/p-1} }+\|\cq u^\ell\|_{\dot B_{2,1}^{d/2-1} } \right)\de t \\
            &+\int_0^T \left( \nu^{-1}\|a^h\|_{\dot B_{p,1}^{d/p} }+\nu^{-1}\|\nu a^\ell\|_{\dot B_{2,1}^{d/2-1} } \right) \de t.
    \end{aligned}
\end{equation}
We now look for an estimate on the high frequencies of $a$: we follow the same energy-type method and we get
\begin{equation}\label{enerestp a}
    \begin{aligned}
       \| \dot\Delta_j a(t)\|_{L^p}+\nu^{-1}\int_0^T \|\dot\Delta_j a\|_{L^p}\de t \leq& \|\dot\Delta_j a_0\|_{L^p}+ \frac{1}{p}\int_0^T \|\dive (u+V)\|_{L^\infty}\|\dot\Delta_j a\|_{L^p}\de t  \\
           &+\int_0^T\|[(u+V)\cdot\nabla,\dot\Delta_j ]a\|_{L^p}\de t+\int_0^T \|\dot\Delta_j (a\dive u+\dive w)\|_{L^p}\de t.
    \end{aligned}
\end{equation}
To bound the last two term above, we apply Lemma \ref{lemmastimaprodLP} and Lemma \ref{commut est Lpsetting}
\begin{itemize}
    \item $\|a\dive u\|_{\dot B_{p,1}^{d/p} }\lesssim \|a\|_{\dot B_{p,1}^{d/p}}\|\dive u\|_{\dot B_{p,1}^{d/p}}$,
    \item $\sum_{j\in\Z}2^{jd/p}\|[(u+V)\cdot\nabla, \dot\Delta_j]a\|_{L^p}\lesssim \|u+V\|_{\dot B_{p,1}^{d/p+1}}\|a\|_{\dot B_{p,1}^{d/p}}$,
\end{itemize}
and, moreover, since we are in the high-frequency regime, it follows that
$$
\|\dive w^h\|_{\dot B_{p,1}^{d/p}}\lesssim \|w^h\|_{\dot B_{p,1}^{d/p+1}}\lesssim \nu^{-2}\|w^h\|_{\dot B_{p,1}^{-2+d/p}}.
$$
Multiplying \eqref{enerestp a} by $2^{\frac{d}{p}j},$ summing over the high frequencies and using the embedding $\dot B_{p,1}^{d/p}(\R^d)\hookrightarrow L^\infty(\R^d)$ together with the estimate for the effective velocity \eqref{est effective velocity}, we obtain: 
\begin{equation}\label{est a w 1}
    \begin{aligned}
        \|\nu a^h\|_{L^\infty\dot B_{p,1}^{d/p} }&+\|a^h\|_{L^1\dot B_{p,1}^{d/p} }+\|w^h\|_{L^\infty\dot B_{p,1}^{d/p-1} }+\nu\|w^h\|_{L^1\dot B_{p,1}^{d/p+1} }\\
        \lesssim & \|\nu a_0^h\|_{\dot B_{p,1}^{d/p} }+\|w_0^h\|_{\dot B_{p,1}^{d/p-1} }+\nu^{-1}\|w^h\|_{L^1\dot B_{p,1}^{d/p-1} }   
        \\
        &+\nu^{-2}\|a^h\|_{L^1\dot B_{p,1}^{d/p} }
        +(\vp+\wp)(\zp+\vp)+\nu^{-1}\yp(\zp+\vp)+\nu^{-1}\xp(\wp+\vp+\yp)+\nu^{-1}\xp\yp\\
        &+\int_0^T\|(\nabla u,\nabla V)\|_{\dot B_{p,1}^{d/p} }\left(\|\cq u^h\|_{\dot B_{p,1}^{d/p-1} }+\|\cq u^\ell\|_{\dot B_{2,1}^{d/2-1} }+\|\nu a^h\|_{\dot B_{p,1}^{d/p}}+\|\nu a^\ell\|_{\dot B_{2,1}^{d/2}} \right)\de t\\
        &+\int_0^T \left(\nu^{-1}\|a^h\|_{\dot B_{p,1}^{d/p}}+\nu^{-1}\|\nu a\|_{\dot B_{2,1}^{d/2+1}} \right)\left(\|\nu a^h\|_{\dot B_{p,1}^{d/p}}+\|a^\ell\|_{\dot B_{2,1}^{d/2-1}} \right)\de t
    \end{aligned}
\end{equation}
Since $\nu\gg1 $, the terms $\nu^{-1}\|w^h\|_{L^1\dot B_{p,1}^{d/p-1}}, \nu^{-2}\|a^h\|_{L^1\dot B_{p,1}^{d/p}}$
can be absorbed into the left-hand side.
Moreover, we use the definition of $w$
\begin{equation*}
    \cq u=w+\nu^{-1}(-\Delta)^{-1}\nabla a,
\end{equation*}
to obtain the estimates
\begin{align*}
    \|\cq u^h\|_{L^\infty\dot B_{p,1}^{d/p-1}}
    \lesssim&\|w^h\|_{L^\infty\dot B_{p,1}^{d/p-1}}+\|\nu a^h\|_{L^\infty\dot B_{p,1}^{d/p}}, \\
\end{align*}
where we used that $2^{-2j}<\nu^2$, and
\begin{align*}
    \nu \|\cq u^h\|_{L^1\dot B_{p,1}^{d/p+1}}
    \lesssim& \|\nu w^h\|_{L^1\dot B_{p,1}^{d/p+1}}+\|a^h\|_{L^1\dot B_{p,1}^{d/p}}.
\end{align*}

Therefore from \eqref{est a w 1} we can conclude that
\begin{equation}
    \begin{aligned}
            \|\nu a&^h\|_{L^\infty\dot B_{p,1}^{d/p}}+\|a^h\|_{L^1\dot B_{p,1}^{d/p}}+\|\cq u^h\|_{L^\infty\dot B_{p,1}^{d/p-1}}
            +\nu\|\cq u^h\|_{L^1\dot B_{p,1}^{d/p+1}}+\|(\cq u_t+\nabla a)^h\|_{L^1\dot B_{p,1}^{d/p-1} }\\
            &\lesssim \|\nu a_0^h\|_{\dot B_{p,1}^{d/p}}+\|w_0^h\|_{\dot B_{p,1}^{d/p-1}} 
             +(\vp+\wp)(\zp+\vp)+\nu^{-1}\yp(\zp+\vp)+\nu^{-1}\xp(\wp+\vp+\yp)\\
        &+\nu^{-1}\xp\yp+\int_0^T \left(\nu^{-1}\|a^h\|_{\dot B_{p,1}^{d/p}}+\nu^{-1}\|\nu a\|_{\dot B_{2,1}^{d/2+1}} \right)\left(\|\nu a^h\|_{\dot B_{p,1}^{d/p}}+\|a^\ell\|_{\dot B_{2,1}^{d/2-1}} \right)\de t\\
         &+\int_0^T\|(\nabla u,\nabla V)\|_{\dot B_{p,1}^{d/p} }\left(\|\cq u^h\|_{\dot B_{p,1}^{d/p-1} }+\|\cq u^\ell\|_{\dot B_{2,1}^{d/2-1} }+\|\nu a^h\|_{\dot B_{p,1}^{d/p}}+\|\nu a^\ell\|_{\dot B_{2,1}^{d/2}} \right)\de t,
    \end{aligned}
\end{equation}
where, analogously to Step 3 in the proof of Theorem \ref{thm:main}, we have used directly the equation \eqref{sist comp} to estimate $\|(\cq u_t+\nabla a)^h\|_{L^1\dot B_{p,1}^{d/p-1} }$ as
\begin{equation*}
    \begin{aligned}
        \|(\cq u_t+\nabla a)^h\|_{L^1\dot B_{p,1}^{d/p-1} }
\lesssim& \|\nu\cq  u^h\|_{L^1\dot B_{p,1}^{d/p+1} }+\|\cq H_2\|_{L^1 \dot B_{p,1}^{-1+d/p} }\\
&+\|\cq ((h+B)\cdot\nabla(h+B))\|_{L^1\dot B_{p,1}^{d/p-1} }+\|\nabla(|h+B|^2)\|_{L^1\dot B_{p,1}^{d/p-1} }.
    \end{aligned}
\end{equation*}

{\em \underline{Step 3} \quad Low-Frequency Estimates for the Compressible Part of \eqref{MHDcomp}.}\\
\\
To bound the low frequencies of $ (a,\cq u)$, we proceed exactly as in Step 3 of Theorem \ref{thm:main} until we get to the inequality \eqref{3.44New}, which we rewrite
\begin{equation*}
\begin{aligned}
     &\mathscr{L}_j(t)+c\text{ min}(\nu 2^{2j},\nu^{-1})\int_0^t \mathscr{L}_j\de \tau\leq \mathscr{L}_j(0)+c\int_0^t \|(\nabla u,\nabla V)\|_{L^\infty}\mathscr{L}_j\de \tau\\
        &\int_0^t \left(\|\dot\Delta_j((h+B)\cdot\nabla(h+B)\|_{L^2}+\|\dot\Delta_j\nabla(|B+h|^2)\|_{L^2}+\|[g_j,f_j,\nu\nabla g_j]\|_{L^2}\right)\de \tau.
\end{aligned}
\end{equation*}
Then, since for low frequencies 
$\text{ min}(\nu 2^{2j},\nu^{-1})=\nu 2^{2j}$, multiplying by $2^{j(d/2-1)} $ and summing over $2^j\nu\leq 1$, we get
\begin{align}
\|(a^\ell,\nu\nabla a^\ell,\cq u^\ell)&\|_{L^\infty \dot B_{2,1}^{d/2-1} }+\|(\nu a^\ell, \nu^2\nabla a^\ell,\nu \cq u^\ell)\|_{L^1 \dot B_{2,1}^{d/2+1} } \nonumber\\
&\lesssim \|(a_0^\ell,\nu\nabla a_0^\ell,\cq u_0^\ell)\|_{L^\infty \dot B_{2,1}^{d/2-1} } +\int_0^T \|(\nabla u,\nabla V)\|_{L^\infty}\|(a^\ell,\nu \nabla a^\ell,\cq u^\ell)\|_{\dot B_{2,1}^{d/2-1}}\de t\nonumber\\
&+\int_0^T \sum_{2^j\nu\leq 1 } 2^{(-d/2+1)j}\|\dot \Delta_j g,\nu\nabla\dot\Delta_j g\|_{L^2}\de t+\int_0^T \sum_{2^j\nu\leq 1 } 2^{(-d/2+1)j}\|\dot \Delta_j \cq f\|_{L^2}\de t\nonumber\\
&+\int_0^T \sum_{2^j\nu\leq 1 } 2^{(-d/2+1)j}\|\dot\Delta_j \cq((g+B)\cdot\nabla(h+B))\|_{L^2}\de t\nonumber\\
&+\int_0^T \sum_{2^j\nu\leq 1 } 2^{(-d/2+1)j}\|\dot\Delta_j\nabla(|h+B|^2)\|_{L^2}\de t.\label{est low comp1}
\end{align}
We now need to estimate the last four terms on the right-hand side of \eqref{est low comp1}. We start by considering the term $\|(\dot\Delta_j g,\nu\nabla \dot\Delta_j g)^\ell\|_{\dot B_{2,1}^{d/2-1} }$: we use the commutator estimate in Lemma \ref{commut est Lpsetting}, together with Lemma \ref{lemmastimaprodLP} and the embedding $\dot B_{2,1}^{d/2-1}\hookrightarrow \dot B_{p,1}^{d/p-1}$, to get that
\begin{equation}\label{est g}
    \begin{aligned}
        \int_0^T \sum_{2^j\nu \leq 1} & 2^{(-d/2+1)j}\|(\dot\Delta_j g,\nu\nabla \dot\Delta_j g)\|_{L^2}\\
        &\lesssim \int_0^T \left(\nu^{-1}\left(\|\nu\cq u^\ell\|_{\dot B_{2,1}^{d/2+1} }+\|\nu\cq u^h\|_{\dot B_{p,1}^{d/p+1} }\right)+\|(\nabla\cp u,\nabla V)\|_{\dot B_{p,1}^{d/p} }\right)\|a^\ell\|_{\dot B_{2,1}^{d/2-1} }\de t\\
        &+\int_0^T \nu^{-1}\left(\|\nu\cq u^\ell\|_{\dot B_{2,1}^{d/2+1}}+\|\nu\cq u^h\|_{\dot B_{p,1}^{d/p+1}}+\nu^{-1}\|(\nabla\cp u, \nabla V)\|_{\dot B_{p,1}^{d/p}}\right)\|\nu a^h\|_{\dot B_{p,1}^{d/p}}\de t.
    \end{aligned}
\end{equation}
We now move to the next terms on the right hand side of \eqref{est low comp1}. Since they are of the form $\cq(ab)$, we use Bony’s decomposition. In particular, if we denote by $\cq^\ell:= \dot S_{j_0+1}\cq $ then, by \eqref{Bony decomp1} and \eqref{Bony decomp2}, we have 
    \begin{equation*}
        \begin{aligned}
            \cq ^\ell(bc)=&\cq^\ell(\dot T_b c+\dot T_c b+R(b,c)) = \cq^\ell(\dot T_b c+R(b,c))+ [\cq^\ell,\dot T_c]b+\dot T_c \cq^\ell b.
        \end{aligned}
    \end{equation*}
Therefore, using Lemma \ref{continuity bony decomp} together with Lemma \ref{lem:commutator 3} one can show that, for all $2\leq p\leq \min \{4,\frac{2d}{d-2} \} $
\begin{equation}\label{est Ql}
    \|\cq(bc)^\ell\|_{\dot B_{2,1}^{d/2-1} }\lesssim \left(\|b\|_{\dot B_{p,1}^{d/p-1}}+\|b^\ell\|_{\dot B_{2,1}^{d/2-1}}\right)\|c\|_{\dot B_{p,1}^{d/p}}.
\end{equation}
This, toghether with Lemma \ref{composit est Lp} and Lemma \ref{commut est Lpsetting} allows us to estimate the term $\|\dot\Delta_j \cq f^\ell\|_{\dot B_{2,1}^{d/2-1} }$ as
\begin{equation}\label{est f}
    \begin{aligned}
        \int_0^T\sum_{j\in\Z}&2^{(-d/2+1)j} \|\dot\Delta_j\cq f^\ell\|_{L^2}\de t\lesssim \nu^{-1} \xp(\vp+\wp)+(\vp+\wp)(\zp+\vp) \\
        &+\nu^{-2}\yp^2(\zp+\vp)+\nu^{-1}\xp\yp+\nu^{-2}\xp^2\yp\\
        &+\int_0^T \left(\nu^{-1}\|a^h\|_{\dot B_{p,1}^{d/p} }+\nu^{-1}\|\nu a^\ell\|_{\dot B_{2,1}^{d/2+1} }\right)\left(\|\nu a^h\|_{\dot B_{p,1}^{d/p} }+\|a^\ell\|_{\dot B_{2,1}^{d/2-1} } \right)\de t\\
        &+\int_0^T \left(\|(\nabla\cp u,\nabla\cq u^h,\nabla V)\|_{\dot B_{p,1}^{d/p} }+\|\nabla\cq u^\ell\|_{\dot B_{2,1}^{d/2} }\right)\left(\|\cq u^\ell\|_{\dot B_{2,1}^{d/2-1} }+\|\cq u^h\|_{\dot B_{p,1}^{d/p-1} } \right)\de t.
    \end{aligned}
\end{equation}
Moreover, using again \eqref{est Ql} and Lemma \ref{lemmastimaprodLP}, we bound the terms involving the magnetic field as follows
\begin{equation}\label{est h,B}
\int_0^T  \sum_{2^j\nu\leq1} 2^{(-d/2+1)j} \left(\|\dot \Delta_j\cq ((h+B)\cdot\nabla(h+B))\|_{L^2}+\|\dot\Delta_j\nabla(|h+B|^2)\|_{L^2}  \right)\de t\lesssim (\vp+\wp)(\zp+\vp).
\end{equation}
Thus, we put \eqref{est g}, \eqref{est f} and \eqref{est h,B} into \eqref{est low comp1}, and we finally get the estimate
\begin{equation}
\begin{aligned}
     \|(a^\ell,\nu\nabla &a^\ell,\cq u^\ell)\|_{L^\infty \dot B_{2,1}^{d/2-1} }+\|(\nu a^\ell,\nu^2\nabla a^\ell,\nu\cq u^\ell)\|_{L^1\dot B_{2,1}^{d/2+1} }\lesssim \|(a_0,\nu\nabla a_0,\cq u_0)^\ell\|_{\dot B_{2,1}^{d/2-1} } \\
     &+\int_0^T \|(\nabla\cp u,\nabla V)\|_{\dot B_{p,1}^{d/p} }\left( \|a^\ell\|_{\dot B_{2,1}^{d/2-1}}+\|\nu a^h\|_{\dot B_{p,1}^{d/p}}  \right)\de t\\
     &+\int_0^T  \left( \nu^{-1}\|a^h\|_{\dot B_{p,1}^{d/p} }+\nu^{-1}\|\nu a^\ell\|_{\dot B_{2,1}^{d/2+1}} \right)\left( \|a^\ell\|_{\dot B_{2,1}^{d/2-1}}+\|\nu a^h\|_{\dot B_{p,1}^{d/p}} \right) \de t\\
     &+\int_0^T  \left( \|\cq u^\ell\|_{\dot B_{2,1}^{d/2+1}}+\|\cq u^h\|_{\dot B_{p,1}^{d/p+1}} \right)\left(\|a^\ell\|_{\dot B_{2,1}^{d/2-1}}+\|\nu a^\ell\|_{\dot B_{2,1}^{d/2}}+\|\nu a^h\|_{\dot B_{p,1}^{d/p}} \right)\de t\\
     &+\int_0^T  \left( \|(\nabla\cp u,\nabla\cq u^h,\nabla V)\|_{\dot B_{p,1}^{d/p}}+\|\nabla \cq u^\ell\|_{\dot B_{2,1}^{d/2}} \right)\left( \|\cq u^\ell\|_{\dot B_{p,1}^{d/p-1}} +\|\cq u^\ell\|_{\dot B_{2,1}^{d/2-1}} \right)\de t\\
         &+\nu^{-1}\xp(\vp+\wp)+\nu^{-1}\xp\yp+(\vp+\wp)(\zp+\vp)+\nu^{-2}\yp^2(\zp+\vp)+\nu^{-2}\xp^2\yp.
\end{aligned}    
\end{equation}
In order to conclude the estimate for the low frequencies in $\xp+\yp$, we use the compressible part of the momentum equation to bound directly $\|(\cq u_t+\nabla a)^\ell\|_{L^1\dot B_{2,1}^{d/2-1} }$. We then combine the estimates for the low frequencies with those for the high frequencies obtained in the previous Step, and applying Gronwall’s lemma, we obtain
\begin{equation}\label{Lpestcomp}
\begin{aligned}
     \xp+\yp\lesssim \ &e^{C\int_0^T \left( \|(\nabla\cp u, \nabla\cq u^h,\nabla V)\|_{\dot B_{p,1}^{d/p} }+\|\nabla\cq u^\ell\|_{\dot B_{2,1}^{d/2}}+\nu^{-1}\|a^h\|_{\dot B_{p,1}^{d/p} }+\nu^{-1}\|\nu a^\ell\|_{\dot B_{2,1}^{d/2+1}}  \right)\de t } \\
     &\times \biggl(\xp(0)+(\vp+\wp)(\zp+\vp)+\nu^{-2}\xp^2\yp+\nu^{-2}\yp^2(\zp+\vp) \\
    &+\nu^{-1}\yp(\zp+\vp)+\nu^{-1}\xp(\wp+\vp+\yp)+\nu^{-2}\zp\wp+\nu^{-3}\xp\yp+\nu^{-2}\vp^2 \biggr).
\end{aligned}   
\end{equation}
{\em \underline{Step 4} \quad Closure of the estimates.}\\
\\
From the estimate for the compressible part \eqref{Lpestcomp}, using that $\nu^{-1}\ll1 $ and \eqref{claim_xp+yp_zp+wp} we have
\begin{equation}\label{est:finale comp lp}
    \begin{aligned}
        \xp+\yp\lesssim &e^{C(\delta+\nu^{-1}D+M )}(\xp(0)+(M+\delta)^2+\nu^{-1}D^2\xp+\nu^{-1}D(\delta+M)\yp\\
        &+\nu^{-1}(M+\delta)\yp+\nu^{-1}(\delta+M+D)\xp).
    \end{aligned}
\end{equation}
We then use the bound \eqref{est:finale comp lp} and the embedding $\dot B_{2,1}^{d/2}\hookrightarrow \dot B_{p,1}^{d/p} $ to estimate the incompressible part as
\begin{equation}\label{est u^2 V^2}
    \begin{aligned}
        \|u\|_{\dot B_{p,1}^{d/p}}^2+ \|V\|_{\dot B_{p,1}^{d/p}}^2\lesssim& \|\cp u\|_{\dot B_{p,1}^{d/p}}^2+ \|\cq u^h\|_{\dot B_{p,1}^{d/p}}^2
        + \|\cq u^\ell\|_{\dot B_{p,1}^{d/p}}^2+\|V\|_{\dot B_{p,1}^{d/p}}^2\\
        \lesssim& 2^j2^{-j} \|\cp u\|_{\dot B_{p,1}^{d/p}}^2+2^j2^{-j}\|\cq u^h\|_{\dot B_{p,1}^{d/p}}^2+2^j2^{-j}\|\cq u^\ell\|_{\dot B_{2,1}^{d/2}}^2
        +2^j2^{-j}\|V\|_{\dot B_{p,1}^{d/p}}^2\\
        \lesssim& \|\cp u\|_{\dot B_{p,1}^{d/p-1}}\|\cp u\|_{\dot B_{p,1}^{d/p+1}}
        +\nu^{-1}\|\cq u^h\|_{\dot B_{p,1}^{d/p-1}}\|\nu\cq u^h\|_{\dot B_{p,1}^{d/p+1}}\\
        &+\nu^{-1}\|\cq u^\ell\|_{\dot B_{2,1}^{d/2-1}}\|\nu\cq u^\ell\|_{\dot B_{2,1}^{d/2+1}}
        +\|V\|_{\dot B_{p,1}^{d/p-1}}\|V\|_{\dot B_{p,1}^{d/p+1}}.
    \end{aligned}
\end{equation}
Inserting \eqref{est u^2 V^2} into \eqref{est zp+wp} we obtain
\begin{equation}
    \begin{aligned}
        \zp+\wp\lesssim &e^{C(M+\nu^{-1}D+\nu^{-1}D^2+\delta^2+M^2+\delta)}(\nu^{-1/2}DM+\nu^{-1}D(M+D+\delta)\wp\\
        &+\nu^{-1}(D+\delta+M)D+\nu^{-1}DM^2.
    \end{aligned}
\end{equation}
Hence, if we assume that
\begin{equation}\label{bound nu^-1D e delta}
    \nu^{-1}D+\nu^{-1}D^2\leq M+M^2 \quad\text{and} \quad \delta\leq\text{max}\{M,1\}
\end{equation}
we have shown that
\begin{equation}
    \begin{aligned}
        &\zp+\wp\lesssim e^{C(M+M^2)}(\nu^{-1/2}MD+\nu^{-1}D(M+D+1)\wp+\nu^{-1}D^2+\nu^{-1}D+\nu^{-1}DM), \\
        &\xp+\yp\lesssim e^{C(M+M^2)}(\xp(0)+M^2+1+\nu^{-1}(M+1+D+D^2+M^2)\yp\\
        &\hspace{1.5cm}+\nu^{-1}(M+1+D+D^2)\xp).       
    \end{aligned}
\end{equation}
Thus, if we further assume that
\begin{equation}\label{hp nu^-1(M+1+D+D^2+M^2)e^}
    \nu^{-1}(M+1+D+D^2+M^2)e^{C(M+M^2)}\ll1
\end{equation}
we obtain
\begin{align}
    &\xp+\yp\lesssim e^{C(M+M^2)}(\xp(0)+M^2+1),\\
    &\zp+\wp\lesssim De^{C(M+M^2)}(\nu^{-1/2}M+\nu^{-1}(M^2+D^2+D+1)).
\end{align}
Then, we define
\begin{align}
    &D:=Ce^{C(M+M^2)}(\xp(0)+M^2+1), \label{defDp}\\
    &\delta:=Ce^{C(M+M^2)}(\xp(0)+M^2+1)(\nu^{-1/2}M+\nu^{-1}(\xp(0)+M^2+1+(\xp(0)+M^2+1)^2)),\label{defDeltap}
\end{align}
and we consider $\nu$ large enough so that
\begin{equation}
    Ce^{C(M+M^2)}(\xp(0)+M^2+1)\leq\sqrt{\nu}.
\end{equation}
With this choice of the constants, the bounds in \eqref{bound nu^-1D e delta} and \eqref{hp nu^-1(M+1+D+D^2+M^2)e^} are satisfied.
Hence, with $D$ and $\delta$ as in \eqref{defDp} and \eqref{defDeltap}, the estimate \eqref{claim_xp+yp_zp+wp} is valid for all $ T<T_*$, and the a priori estimates on $a,\ v,\ b$ imply that the solution exists globally in time thanks to Theorem \ref{thm:continuation}. This concludes the proof.
\end{proof}

\section{Applications to Magnetic Reconnection}
\label{sec:magnetic reconnection}
In this section, we use our result to construct examples of solutions of the compressible MHD system that exhibit magnetic reconnection. The latter refers to a change in the topological structure of the integral lines of the magnetic field $b$. Specifically, given $t\in [0,T]$, we define a {\em magnetic line} $\gamma_t$ as any solution of the following (autonomous) system of ODEs
\begin{equation}
    \label{def:magnetic line}
    \frac{\de}{\de s}\gamma_t(s)=b(t,\gamma_t(s)).
\end{equation}

\begin{definition}\label{def:rec}
A (sufficiently regular) solution $(\rho, v,b,p)$ of \eqref{MHDcomp} 
shows magnetic reconnection if there exist $t_1, t_2\in[0,T]$ such that there is no homeomorphism of $\R^3$ mapping the set of integral lines of $b(t_1,\cdot)$ into integral lines of $b(t_2,\cdot)$.
\end{definition}
In the non-resistive case, i.e. $\eta=0$, Alfven's Theorem states that the integral lines of a sufficiently smooth magnetic field are transported by the fluid. In particular, the topology of the integral lines of the magnetic field is frozen under evolution.
This is based on the following observation: a smooth solution $b$ of the PDE
\begin{equation}\label{eq:alfven}
\begin{cases}
\partial_t b+(v\cdot \nabla) b=  (b\cdot\nabla)v,\\
b_{|{t=0}}=b_0,
\end{cases}
\end{equation}
satisfies the formula
\begin{equation}\label{eq:alfven2}
b(t,\Phi_t(x))=\nabla \Phi_t(x)b_0(x),
\end{equation}
where $\Phi_t$ is the flow of $v$, i.e.
\begin{equation}
\begin{cases}
\frac{\de}{\de t}\Phi_t(x)=v(t,\Phi_t(x)),\\
\Phi_0(x)=x.
\end{cases}
\end{equation}
The formula \eqref{eq:alfven2} implies that the magnetic lines of $b_0$ are transported by the fluid flow. Since $v$ is assumed to be smooth, $\Phi_t$ is a diffeomorphism and then at every time $t>0$, the integral lines of $b_0$ and $b(t,\cdot)$ are diffeomorphic. Thus, reconnection is forbidden.

In contrast, in the resistive case ($\eta > 0$) the topology of magnetic field lines can change: heuristically, magnetic diffusion allows for the breaking of topological rigidity. This phenomenon is observed both experimentally and numerically, see \cite{MagnPhys, Priest}, and recently rigorous mathematical examples have been provided for the incompressible model \eqref{MHDincomp}, see \cite{CCL, CL2, LP}.

In order to apply our result to the construction of solutions showing magnetic reconnection, it is convenient to first recall the main results of \cite{LP}. In that work, the authors introduce a condition on the initial data that ensures the existence of global strong solutions to the three-dimensional incompressible MHD system \eqref{MHDincomp}. 
Moreover, under this condition, they are able to exhibit an explicit class of initial configurations leading to magnetic reconnection. A distinctive feature of their framework is that the admissible data may be large in every critical space. 
The construction relies on a formulation in Els\"asser variables, in the spirit of \cite{HHW}, although the condition in \cite{LP} is of a different nature (in particular, the class of initial data they consider for the reconnection does not satisfy the smallness condition in \cite{HHW}).

Throughout this section, we restrict our attention to the case $\mu = \eta$. This simplification is physically relevant in astrophysical contexts, where both the Reynolds number and the magnetic Reynolds number are extremely large. In such regimes, the difference between the viscous and magnetic diffusivities is negligible, and taking $\mu = \eta$ does not qualitatively affect the dynamics of the system. 
Once global solutions to the incompressible system are available, Theorem \ref{thm:main} allows us to obtain corresponding global solutions to the compressible MHD equations \eqref{MHDcomp}, provided that the bulk viscosity parameter $\lambda$ is sufficiently large. Moreover, the compressible solutions remain close to their incompressible counterparts for all times.

To establish the occurrence of magnetic reconnection in the compressible setting, we then exploit the structural stability of the reconnection mechanism exhibited in \cite{LP}. This yields compressible solutions that inherit the reconnection properties of the incompressible ones.

We now briefly describe the results contained in \cite{LP}. First, the authors prove the following result (we use different notations). 
\begin{theorem}
Let $V_0, B_0\in H^r(\R^3)$ for some $r\geq 3$ be two divergence-free vector fields such that
$$
\|V_0-B_0\|_{H^r}\lesssim \e, \quad \mbox{and} \quad \|e^{\eta t\Delta}V_0\|_{L^2_tW^{k,\infty}_x}+\|e^{\eta t\Delta}B_0\|_{L^2_tW^{k,\infty}_x}\lesssim 1+ N^k, \quad\forall\, 1\leq k\leq r,
$$
for some constant $N\geq 1$ and some (small) constant $0<\e\leq cN^{-r-1}$ with $c>0$ sufficiently small. Then there exists a unique global smooth solution $(V,B)$ of \eqref{MHDincomp} with $\eta=\mu$ that satisfies
$$
\|V(t,\cdot)-e^{\eta t\Delta}V_0\|_{H^k}+\|B(t,\cdot)-e^{\eta t\Delta}B_0\|_{H^k}\leq C\e N^k, \quad \forall 0\leq k\leq r,\quad \forall t\geq 0.
$$
\end{theorem}

A choice of initial data which satisfies the above assumptions is the following:
\begin{equation}
    V_0=M\curl (\phi B_N),\qquad B_0=M\curl (\phi B_N)+\e \curl w, 
\end{equation}
where $M>0$, $\phi, w\in H^r(\R^3)$ and $B_N$ is a Beltrami field with frequency $N>1$, namely a solution of
$$
\curl B_N= NB_N.
$$
Then, one can show that (see \cite{CL})
$$
\|V_0\|_{\dot{B}^{-1}_{\infty,\infty}}\simeq\|B_0\|_{\dot{B}^{-1}_{\infty,\infty}}\simeq M,
$$
implying that the initial datum is large in any critical space ($M$ is arbitrary). \\

The second main result of \cite{LP} is the following.
\begin{theorem}\label{thm:renato claudia}
Given any constants $M,T>0$ there exist two smooth
(finite energy) divergence-free vector fields $(V_0,B_0)$ in $\R^3$ such that $\|V_0\|_{\dot{B}^{-1}_{\infty,\infty}}\simeq \|B_0\|_{\dot{B}^{-1}_{\infty,\infty}}\simeq M$ and \eqref{MHDincomp} admits a unique global smooth solution $(V,B)$ with initial datum $(V_0,B_0)$, such that the
magnetic lines at time $t=0$ and $t=T$ are not topologically equivalent, meaning that there is no homeomorphism of $\R^3$ into itself mapping the magnetic lines of $B(0,\cdot)$ into those of $B(T,\cdot)$.
\end{theorem}
In order to prove the reconnection result, the authors of \cite{LP} make the particular choice of
\begin{equation}\label{def:beltrami riconnessione}
B_0=M\curl (\phi B_N)+\e e^{-\eta T \Delta}\curl(\psi W),
\end{equation}
where the parameters are chosen so that
\begin{itemize}
    \item $M>0$ is the size of the initial datum;
    \item $T>0$ is the target reconnection time;
    \item $0<\e\ll 1$ is the size of the difference $V_0-B_0$;
    \item $N>1$ is a large parameter.
\end{itemize}
Concerning the vector fields,
\begin{itemize}
    \item[$(i)$] $B_N$ is a high frequency Beltrami field with no zeros;
    \item[$(ii)$] $W$ is a smooth bounded vector field that has some null points;
    \item[$(iii)$] $\phi$ and $\psi$ are smooth functions that decrease polinomially and exponentially fast at infinity;
    \item[$(iv)$] $B_0$ has no equilibrium points;
    \item[$(v)$] $\curl(\psi W)$ has at least one hyperbolic equilibrium.
\end{itemize}
Note that the properties $(iv)$ and $(v)$ are robustly stable under $C^1$-perturbations. Hence, showing that the solution $b(T,\cdot)$ is close to $\curl(\psi W)$ at time $T$ implies that equilibria must have formed at some intermediate time. We also remark that zeros of the magnetic field are commonly observed as reconnection sites, see \cite{MagnPhys, Priest}, and the loss/creation of equilibrium points indicates a change in the magnetic topology. \\

Having recalled the incompressible construction of \cite{LP}, we now explain how our Theorem \ref{thm:main} applies to this setting. 
First of all, in order to apply the stability result in Theorem \ref{thm:main}, we need to verify that the class of initial data constructed in Theorem \ref{thm:renato claudia} belongs to the (smaller) space $\dot B^\frac12_{2,1}(\R^3)\subset \dot B^{-1}_{\infty,\infty}(\R^3)$.
In \cite{CL} it is shown that this class of vector fields satisfy the following
\begin{equation}\label{bound H^r}
 \|V_0\|_{H^r,}   \|B_0\|_{H^r}\leq C_r,
\end{equation}
for all $r\geq 0$. Thus, by Proposition \ref{prop:real interpolation} we have that $V_0,B_0\in\dot B^\frac12_{2,1}(\R^3)$.
Thus, the global solutions built in \cite{LP}, together with their reconnection mechanism, provide a suitable reference dynamics for the singular limit considered in this work. By Theorem~\ref{thm:main}, compressible solutions with sufficiently large bulk viscosity remain uniformly close to these incompressible solutions for all times, and in particular inherit their qualitative features. As a consequence, the reconnection scenario established in \cite{LP} persists in the compressible regime. This is formalized in the following corollary, whose proof follows directly from the stability estimate in Theorem~\ref{thm:main} and the structural properties of the solutions constructed in \cite{LP}.

\MainTheoremC*

\begin{remark}
Note that once the global solution is constructed in the critical space, one can propagate additional Sobolev regularity of the initial data. Indeed, the velocity field enjoy the regularity
$$
\nabla^2 v \in L^1((0,+\infty);\dot B^{1/2}_{2,1}(\R^3)),
$$
which is equivalent to
$$
v \in L^1((0,+\infty);\dot B^{5/2}_{2,1}(\R^3))
\hookrightarrow
L^1((0,+\infty);W^{1,\infty}(\R^3)).
$$
As a consequence, the Lipschitz regularity allows one to control the transport terms in the magnetic equation and to close standard energy estimates for higher Sobolev norms. In particular, the regularity of the initial datum $b_0$ is preserved and
interpolating between the critical Besov space $\dot B^{1/2}_{2,1}$ and a higher Sobolev space one can control the difference $\|b(t,\cdot)-B(t,\cdot)\|_{C^1}$, ensuring that the two vector fields remain close in $\mathcal C^1$, thus preserving the associated geometric structures $(iv)$ and $(v)$.
\end{remark}

\subsection*{Acknowledgements}
The authors are partially supported by INdAM-GNAMPA and by the projects PRIN2020 ``Nonlinear evolution PDEs, fluid dynamics and transport equations: theoretical foundations and applications” and PRIN2022 ``Classical equations of compressible fluid mechanics: existence and properties of non-classical solutions''.


\end{document}